\documentclass[11pt]{article}

\usepackage{amsfonts,amssymb,stmaryrd,amscd,amsmath,latexsym,amsbsy}
\usepackage{pstricks,pst-plot,pst-node}
\usepackage{a4wide}

\usepackage{amsmath,amsthm,amsfonts}

\theoremstyle{plain}

\newtheorem{theorem}{Theorem}[section]

\newtheorem{corstar}{Corollary*}[section]

\newtheorem{prop}[theorem]{Proposition}

\newtheorem{lemma}[theorem]{Lemma}

\newtheorem{cor}[theorem]{Corollary}

\newtheorem{theoremstar}{Theorem*}[section]

\newtheorem*{sol}{Solution}

\newtheorem*{conj}{Conjecture}

\theoremstyle{definition}

\theoremstyle{remark}

\newcommand{\CC}{\mathbb C}

\newcommand{\ZZ}{\mathbb Z}

\newcommand{\solu}[1]{\begin{sol}{\bf (\ref{#1})}}

\newcommand{\gotg}{\widehat{\mathfrak {g}}}

\newcommand{\glh}{\hat{\mathfrak{gl}}(n+1)}

\newcommand{\An}{\mathcal{A}_n}

\newcommand{\pp}{\mathfrak{p}}

\newcommand{\QQ}{\mathbb{Q}}

\newcommand{\Hb}{\mathrm{Hilb}}

\newcommand{\DT}{\mathrm{DT}}

\title{Donaldson-Thomas theory of $\An\times\mathbf{P}^1$}
\author{Davesh Maulik and Alexei Oblomkov}

\begin{document}

\maketitle

\begin{abstract}
We study the relative Donaldson-Thomas theory of $\An\times\mathbf{P}^{1}$, where $\An$ 
is the surface resolution of type $A_n$ singularity.  The action of divisor operators in the theory is expressed in terms of operators of
the affine algebra $\glh$ on Fock space.  Assuming a nondegeneracy conjecture, this gives a complete solution for the theory.  
The results complete the comparison of this theory with
the Gromov-Witten theory of $\An\times\mathbf{P}^{1}$
and the quantum cohomology of the Hilbert scheme of points on $\An$.
\end{abstract}

\tableofcontents

\section{Introduction}

   Let
$\zeta$ be a primitive $(n+1)$-th root of unity and consider the action of the 
cyclic group $\mathbb{Z}_{n+1}$ on $\mathbb{C}^2$ where the generator acts via
$$(z_1,z_2) \mapsto (\zeta z_1,\zeta^{-1} z_2).$$
Let $\An$ be the minimal resolution of the quotient
$$\An \rightarrow \mathbb{C}^2/\mathbb{Z}_{n+1}.$$
The diagonal action of $T= (\mathbb{C}^{\ast})^{2}$ on
$\mathbb{C}^2$ commutes with the cyclic group action and
therefore lifts to a $T$-action on $\An$.

In this paper, we study the $T$-equivariant Donaldson-Thomas theory of the threefold
$$X = \An \times \mathbf{P}^{1}$$
relative to fibers of $X$ over $\mathbf{P}^{1}$.
Donaldson-Thomas invariants are defined by integration over the moduli space 
of ideal sheaves
of $X$. 
In the presence of relative fibers, we consider ideal sheaves with boundary conditions along these fibers.  Although $X$ is quasi-projective, these integrals can be well-defined via equivariant residue.

The main result of this paper is an explicit understanding of specific relative DT invariants in
terms of the action of the affine algebra $\widehat{\mathfrak{gl}}(n+1)$ on its basic representation.
Under the assumption of a nondegeneracy conjecture (see section \ref{generation conjecture}),
this will give a complete evaluation of the relative DT theory of $X$.  

The Gromov-Witten/Donaldson-Thomas conjecture predicts a precise equivalence between the DT theory of a threefold and its Gromov-Witten theory, defined using moduli spaces of stable maps.  Using previous work \cite{gwan, hilban}, we are able to prove the GW/DT conjecture for the relative invariants studied here.  These calculations provide the starting point for the proof \cite{MOOP}
of the correspondence in the case of an arbitrary toric threefold.  Along with the case of $\CC^{2}$ studied in \cite{okpandt}, these surfaces are the only nontrivial cases for which the relative DT theory has been even partially calculated.

\subsection{Definitions}

Given $\beta \in H_{2}(\An,\ZZ)$ and integers $m, \chi \in \ZZ$,
the moduli space of ideal sheaves
$$I_{\chi}(X, (\beta,m))$$
parametrizes ideal sheaves of proper 
subschemes $Z \subset X$ of dimension at most $1$
with
$$-c_{2}(\mathcal{O}_{Z}) = (\beta, m) \in H_{2}(\An \times \mathbf{P}^1,\ZZ) = H_{2}(\An,\ZZ)\oplus \ZZ$$
and $$\chi(\mathcal{O}_Z) = \chi.$$  The moduli space carries a 
$T$-equivariant perfect obstruction theory obtained from the deformation theory of ideal sheaves.  Although it is noncompact, its $T$-fixed locus is compact.

The fiber of $X$ over a point $z \in \mathbf{P}^{1}$ determines a $T$-equivariant divisor
$$\An \times z \subset X.$$
Given $k$ distinct marked points $z_1, \dots, z_k \in \mathbf{P}^{1}$, we will consider the residue theories of $X$ relative to the divisor $S$ given by the disjoint union of the fibers over these points
$$S = \cup \An \times z_{i}.$$

Let $I_{\chi}(X/S, (\beta,m))$ denote the relative moduli space of ideal sheaves, which parametrizes
ideal sheaves on $X$ which intersect the divisor $S$ transversely; the space is partially compactified by allowing $X$ to degenerate along components of $S$.   Under the degeneration of the base $\mathbf{P}^{1}$ to a broken $\mathbf{P}^{1}$, there is a formula relating the DT theory of $X$ to the relative DT theory of each component of the degeneration.  
 See \cite{mnop2} for a more detailed discussion of the geometry of this moduli space and the degeneration formula.  The moduli space has been constructed in \cite{baosenwu} and 
 the foundational details of the obstruction theory and degeneration formula have been announced
 in forthcoming work of J. Li and B. Wu.

For each relative fiber over $z_i$, we impose relative conditions as follows.  A cohomology-weighted partition of $m$ is an unordered collection of pairs
$$\overrightarrow{\mu}=\{(\mu^{(1)}, \gamma_{1}), \dots, (\mu^{(l)}, \gamma_{l})\},$$
where $\mu^{(1)}, \dots, \mu^{(l)}$ is a partition of $m$ with labels $\gamma_{j} \in H_{T}^{*}(\An,\QQ)$.
It follows from Nakajima's construction of $H_{T}^{*}(\Hb_{m}(\An),\QQ)$ that we can naturally associate to $\overrightarrow{\mu}$ a $T$-equivariant cohomology class
$$\overrightarrow{\mu} \in H_{T}^{\ast}(\Hb_{m}(\An),\QQ),$$
in the Hilbert scheme of points of $\An$.  

Suppose we are given $k$ cohomology-weighted partitions
$\overrightarrow{\mu_{1}}, \dots, \overrightarrow{\mu_{k}},$
associated to the relative fibers over $z_{1},\dots, z_{k}$.  For each $z_i$, we have a boundary map
$$\epsilon_i: I_{\chi}(X/S, (\beta,m)) \rightarrow \Hb_{m}(\An)$$
to the Hilbert scheme of $m$ points on $\An$, obtained by sending the ideal sheaf of $Z$ to the
subscheme $Z \cap (\An\times z_i)$.   We can use these maps to define relative conditions by pulling back the cohomology class $\overrightarrow{\mu_{i}}$ on $\Hb_{m}(\An)$ via the map $\epsilon_{i}$.

If we fix a curve class $(\beta,m)\in H_{2}(X,\ZZ)$, the relative DT partition function of $X/S$ is defined to be
$$
\begin{aligned}
\mathsf{Z}_{\DT}(X)_{(\beta,m), \overrightarrow{\mu_{1}}, \dots, \overrightarrow{\mu_{k}}} &=
\sum_{\chi \in \ZZ} q^{\chi} \int_{[I_{\chi}(X/S,(\beta,m)]^{vir}} \prod_{i=1}^{k} \epsilon_{i}^{*}(\overrightarrow{\mu_{i}})\\
  &= \sum q^{\chi} 
\int_{[I_{\chi}(X/S,(\beta,m)^{T}]^{vir}} \frac{\prod_{i=1}^{k} \epsilon_{i}^{*}(\overrightarrow{\mu_{i}})}
{e(\mathrm{Norm}^{vir})}
\end{aligned}
$$
Since $I_{\chi}(X/S)$ is noncompact, the first expression is not well-defined, so we instead use a
residue integral on the fixed locus (which is compact) as a definition.  Here, $e(\mathrm{Norm}^{vir})$
is the Euler class of the virtual normal bundle to the fixed locus.

We can further sum over possible values of $\beta$ to yield the function
$$\mathsf{Z}_{\DT}(X)_{\overrightarrow{\mu_{1}},\dots, \overrightarrow{\mu_{k}}} = 
\sum_{\beta} s^{\beta} \mathsf{Z}_{\DT}(X)_{(\beta,m), \overrightarrow{\mu_{1}}, \cdots, \overrightarrow{\mu_{k}}}
\in \QQ(t_1,t_2)((q))[[s_{1},\dots,s_{n}]],$$
where
$$s^{\beta} = s_{1}^{(\beta,\omega_{1})}\cdot \dots\cdot s_{n}^{(\beta,\omega_{n})}$$
and $\omega_{1}, \dots, \omega_{n} \in H^{2}(\An,\QQ)$ is a dual basis to $E_{1},\dots,E_{n}$.

Finally, we will largely be interested in the partition function obtained by normalizing these invariants with respect to the $\beta=0, m=0$ partition function:
$$
\mathsf{Z}_{\DT}^{\prime}(X)_{\overrightarrow{\mu_{1}}, \dots, \overrightarrow{\mu_{k}}} =
\frac{\mathsf{Z}_{\DT}(X)_{\overrightarrow{\mu_{1}}, \dots, \overrightarrow{\mu_{k}}}}
{\mathsf{Z}_{\DT}(X)_{(0,0),\emptyset, \dots, \emptyset}} \in \QQ(t_1,t_2)((q))[[s_1,\dots,s_n]].
$$
The relative degree $0$ series has already been computed in \cite{mnop2, okpandt}
$$\mathsf{Z}_{DT}(X)_{(0,0), \emptyset, \dots, \emptyset} = M(-q)^{(k-2)\frac{(t_1+t_2)^{2}}{(n+1)t_1t_2}}$$
where 
$$M(q) = \prod_{n=1}^{\infty}\frac{1}{(1-q^{n})^{n}}$$
is the MacMahon function.
\subsection{GW/DT correspondence}

In what follows, let
$$\mathsf{Z}_{GW}^{\prime}(X)_{\overrightarrow{\mu_{1}}, \dots, \overrightarrow{\mu_{k}}}\in \QQ(t_1,t_2)((u))[[s_{1},\dots,s_{n}]]$$
denote the reduced partition function for the relative Gromov-Witten theory of $X$ relative to $k$ fibers.
We refer the reader to \cite{gwan} for notation and results for this theory.

The main theorem of this paper is the following case of the GW/DT correspondence for $X$:

\begin{theorem}\label{gw dt theorem}
Given cohomology-weighted partitions $\overrightarrow{\mu}, \overrightarrow{\nu}$ and
$$\overrightarrow{\rho} = (1,1)^{m}, (2,1)(1,1)^{m-1},\quad \mathrm{or}\quad (1,\omega_{i})(1,1)^{m-1}$$
the partition function
$\mathsf{Z}^{\prime}_{\DT}(X)_{\overrightarrow{\mu}, \overrightarrow{\rho},\overrightarrow{\nu}}$
is a rational function of $q, s_{1}, \dots, s_{n}$.
After the change of variables $q = -e^{iu}$,
$$(-iu)^{l(\mu)+l(\rho)+l(\nu)-m}\mathsf{Z}^{\prime}_{GW}(X)_{\overrightarrow{\mu}, \overrightarrow{\rho},\overrightarrow{\nu}}
=(-q)^{-m}\mathsf{Z}^{\prime}_{\DT}(X)_{\overrightarrow{\mu}, \overrightarrow{\rho},\overrightarrow{\nu}}.$$
\end{theorem}

Both the rationality and the equivalence with Gromov-Witten theory are of course valid for fixed $\beta \in H_{2}(\An)$ and, in that form, are expected to hold for arbitrary threefolds.
The rationality in $s_{1},\dots, s_{n}$ is a special feature of this geometry.

The condition on $\overrightarrow{\rho}$ is equivalent to asking $\overrightarrow{\rho}$ to equal
$1$ or a divisor class in $H_{T}^{2}(\Hb_m(\An),\QQ)$.  We will give precise expressions for the above partition functions in terms of operators in the affine algebra $\widehat{\mathfrak{gl}}(n+1)$.  Under the assumption of a nondegeneracy conjecture in section~\ref{generation conjecture} on these operators, we can remove
the hypothesis on $\overrightarrow{\rho}$.

\begin{theoremstar}\label{gw dt with conjecture}
Under the assumption of section~\ref{generation conjecture}, the statement of theorem \ref{gw dt theorem} is true for the partition function $$\mathsf{Z}_{\DT}^{\prime}(X)_{\overrightarrow{\mu_{1}}, \dots, \overrightarrow{\mu_{k}}}$$
with arbitrary relative conditions $\overrightarrow{\mu_{1}},\dots, \overrightarrow{\mu_{k}}$.
\end{theoremstar}

\subsection{Relation to Quantum Cohomology of $\Hb(\An)$}

For $\overrightarrow{\mu},\overrightarrow{\nu},\overrightarrow{\rho}\in H_{T}^{*}(\Hb_{m}(\An),\QQ)$,
the genus $0$, $3$-pointed, $T$-equivariant Gromov-Witten invariants of
$\Hb_{m}(\An)$ is defined by the following sum over curve classes
$$
\langle \overrightarrow{\mu}, \overrightarrow{\rho},\overrightarrow{\nu}\rangle^{\Hb}
=\sum_{\beta \in H_{2}(\Hb(\An),\ZZ)} \langle \overrightarrow{\mu}, \overrightarrow{\rho},\overrightarrow{\nu}\rangle^{\Hb}_{0,3,\beta}
q^{D\cdot \beta}\prod s_{i}^{(1,\omega_i)\cdot\beta}
$$
Here, we encode the degree of $\beta$ with respect to the basis of divisors given by the cohomology-weighted partitions
$$D = -(2,1)(1,1)^{m-2}, \quad (1,\omega_{i})=(1,\omega_{i})(1,1)^{m-1}.$$
The results here, along with our earlier paper \cite{hilban} yield the Donaldson-Thomas/Hilbert correspondence for the family of $\An$ surfaces.

\begin{theorem}\label{hilb dt theorem}
For $\overrightarrow{\rho} \in H_{T}^{2}(\Hb_{m}(\An),\QQ),$ or $\overrightarrow{\rho}= (1)^{m}$, we have
$$Z_{\DT}^{\prime}(X)_{\overrightarrow{\mu},\overrightarrow{\rho},\overrightarrow{\nu}} =
q^{m} \langle \overrightarrow{\mu},\overrightarrow{\rho},\overrightarrow{\nu}
\rangle^{\Hb}.$$
Under the assumption of section~\ref{generation conjecture}, the statement holds for all $\overrightarrow{\rho}$.
\end{theorem}

This yields the following triangle of equivalences:

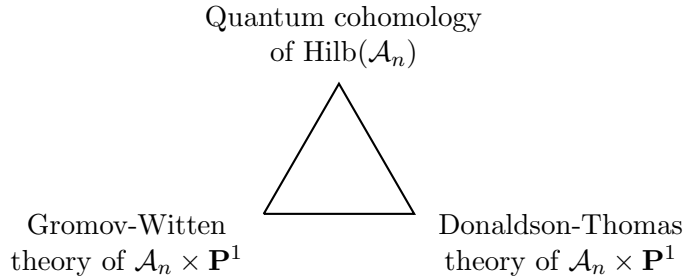
\begin{figure}[hbtp]\psset{unit=0.5 cm}
  \begin{center}
    \begin{pspicture}(-6,-2)(10,6)
    \psline(0,0)(2,3.464)(4,0)(0,0)
    \rput[rt](0,0){
        \begin{minipage}[t]{3.64 cm}
          \begin{center}
            Gromov-Witten \\ theory of $\An \times \mathbf{P}^1$
          \end{center}
        \end{minipage}}
    \rput[lt](4,0){
        \begin{minipage}[t]{3.64 cm}
          \begin{center}
             Donaldson-Thomas\\ theory of $\An \times \mathbf{P}^1$
          \end{center}
        \end{minipage}}
    \rput[cb](2,4.7){
        \begin{minipage}[t]{4 cm}
          \begin{center}
           Quantum cohomology \\ of $\Hb(\An)$
          \end{center}
        \end{minipage}}
    \end{pspicture}
  \end{center}
\end{figure}

In the case of $\CC^{2}$ this triangle has been established in the papers \cite{localcurves},\cite{okpanhilb}\cite{okpandt}.  
While the
equivalence between Gromov-Witten theory and Donaldson-Thomas
theory is expected to hold for arbitrary threefolds, the
relationship with the quantum cohomology of the Hilbert scheme
breaks down for a general surface, at least in the specific form
we describe here.  Our work for $\An$ surfaces provides the only
other examples for which this triangle is known to hold.

\subsection{Overview}
This paper follows a strategy based on our earlier work \cite{hilban}, motivated by the correspondence with quantum cohomology.  The divisor operators are evaluated by proving a sequence of partial evaluations which uniquely determine them via a reconstruction statement. The key idea is to exploit the partial geometric relationship between ideal sheaves on $\An\times\mathbf{P}^{1}$ and rational curves on $\Hb(\An)$, using techniques first developed in \cite{okpandt, okpanhilb} for the case of $\CC^{2}$. Although these moduli spaces are quite different, we show that the calculation of specific DT invariants can be reduced to the quantum cohomology invariants evaluated in \cite{hilban} and that these imply the desired result in general.  In the case of both $\CC^{2}$ and $\An$, it is the existence of a holomorphic symplectic form and the resulting $t_1+t_2$-divisibility statements that enable us to make a direct geometric comparison between the distinct geometries. Because of this comparison, most of the arguments in this paper are schematic. In particular, as we plan to study later, they should apply with minor modification to the stable pairs theory introduced in \cite{pt}. For the sake of brevity, arguments and results from these two earlier papers will just be cited in this paper when they apply without modification.


\subsection{Acknowledgements}
We wish to thank J. Bryan, P. Etingof, A. Okounkov and R. Pandharipande for many useful discussions.
D.M. was partially supported by an NSF Graduate Fellowship and a Clay Research Fellowship.
A.O. was partially supported by NSF grant DMS-0111298 and DMS-0701387.

\section{Rubber theory of $\An\times\mathbf{P}^{1}$}

The proof of theorem \ref{gw dt theorem} will proceed by direct evaluation of the partition functions.  It turns out these can be expressed in terms of nonrigid relative invariants, or rubber integrals, which will be the main focus of this paper.  In this section, 
after some preliminaries, we introduce these integrals and give the explicit formula for their evaluation.

\subsection{Notation}

We first set notation for the geometry of the $\An$ surfaces.
The divisor classes of the rational curves $E_{1},\dots, E_{n}$
span $H^{2}(\An,\QQ)$ and, along with the identity class, span the full 
cohomology ring of $\An$.  We will also work with the dual basis
$\{\omega_1,\dots,\omega_n\}$ of $H^{2}(\An,\QQ)$, defined by the
property that
$$\langle\omega_{i},E_{j}\rangle = \delta_{i,j}$$
under the Poincare pairing.

Under the $T$-action, there are $n+1$ fixed points $p_1, \dots,
p_{n+1}$; the tangent weights at the fixed point $p_i$ are given
by
\begin{gather*}
w^L_i:=(n+2-i)t_1+(1-i)t_2,\\
w^R_i:=(-n+i-1)t_1+it_2.
\end{gather*}
The $E_i$ are the $T$-fixed curves joining $p_{i}$ to $p_{i+1}$.
We denote by $E_{0}$ and $E_{n+1}$ for the noncompact $T$-fixed
curve direction at $p_{1}$ and $p_{n+1}$ respectively.
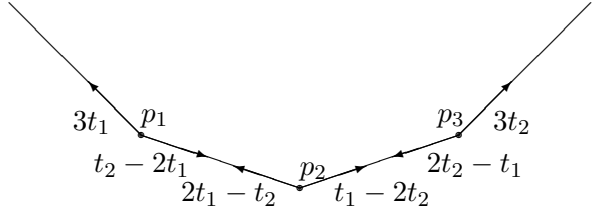
\begin{center}
\begin{picture}(300,130)(0,0) \put(150,30){\line(-3,1){60}}
\put(90,50){\line(-1,1){50}} \put(64,52){$3t_1$}
\put(90,50){\vector(-1,1){20}} \put(90,50){\vector(3,-1){25}}
\put(150,30){\vector(-3,1){25}} \put(72,36){$t_2-2t_1$}
\put(105,25){$2t_1-t_2$} \put(90,50){\circle{2}}
\put(90,50){\circle{1}} \put(90,55){$p_1$}
\put(150,30){\circle{1}}\put(150,30){\circle{2}}
\put(150,35){$p_2$}
\put(150,30){\line(3,1){60}}\put(150,30){\vector(3,1){25}}
\put(210,50){\vector(-3,-1){25}} \put(210,50){\line(1,1){50}}
\put(210,50){\vector(1,1){20}}\put(210,50){\circle{1}}
\put(210,50){\circle{2}} \put(202,55){$p_3$} \put(223,52){$3t_2$}
\put(163,25){$t_1-2t_2$}\put(198,36){$2t_2-t_1$}
\end{picture}
\end{center}

There is an identification of lattices between $H_{2}(\An, \mathbb{Z})$
with the $A_n$ root lattice obtained by sending the exceptional
curves $E_{i}$ to the simple roots $\alpha_{i,i+1}$. Under this
identification, there is a distinguished set of effective curve
classes
$$\alpha_{ij} = E_{i}+ \dots + E_{j-1}$$
which correspond to positive roots in the $A_{n}$ lattice.

\subsection{Rubber geometry}\label{rubber geometry}

Given discrete invariants $\beta, m, \chi$, the rubber moduli space of ideal sheaves
$$I_{\chi}(X, (\beta,m))^{\sim}$$
parametrizes ideal sheaves on $\An\times\mathbf{P}^1$, relative to the fibers over $0$ and $\infty$,
up to equivalence given by $\CC^{\ast}$-scaling along the $\mathbf{P}^1$ direction.
That is, our definition is nearly identical to the usual DT theory of $X$ relative to two fibers, except that now ideal sheaves are considered equivalent if they are isomorphic after applying an automorphism of $\mathbf{P}^{1}$ fixing $0$ and $\infty$.  This moduli space occurs naturally as components of the boundary compactification of the relative moduli space from before.

As before, $I_{\chi}(X,(\beta,m))^{\sim}$ inherits a $T$-action and a $T$-equivariant perfect obstruction theory  of rank $2m -1.$  Relative conditions are again obtained by boundary maps $\epsilon_{0},\epsilon_{\infty}$ to $\Hb_{m}(\An)$.
Although it is noncompact, invariants can again be defined via residue, as the $T$-fixed locus is compact.

We use angle bracket notation for the following rubber DT invariants, which will be the main focus of this paper.  Given two cohomology-weighted partitions of $m$, we have
$$\langle \overrightarrow{\mu}, \overrightarrow{\nu}\rangle^{\DT,\sim}_{\beta,\chi}
= \int_{[I_{\chi}(X,(\beta,m))^{T,\sim}]^{vir}} \frac{\epsilon_{0}^{*}(\overrightarrow{\mu})\epsilon_{\infty}^{*}
(\overrightarrow{\nu})}{e(\mathrm{Norm}^{vir})}.$$
We suppress the $m$ in our notation, since it can be deduced from $\overrightarrow{\nu}$.

We have the partition function
$$\langle \overrightarrow{\mu}, \overrightarrow{\nu}\rangle^{\DT,\sim}_{\beta}
= \sum_{\chi} q^{\chi}\langle \overrightarrow{\mu}, \overrightarrow{\nu}\rangle^{\DT,\sim}_{\beta,\chi}.$$

Since the case of $\beta=0$ can be easily evaluated using the work of \cite{okpandt}, we will mainly focus on the sum of contribution from nonzero $\beta$:
$$\langle \overrightarrow{\mu}, \overrightarrow{\nu}\rangle^{\DT,\sim}_{+} 
= \sum_{\beta\ne 0,\chi} q^{\chi}s^{\beta}
\langle \overrightarrow{\mu}, \overrightarrow{\nu}\rangle^{\DT,\sim}_{\beta,\chi}
\in \QQ(t_1,t_2)((q))[[s_{1},\dots,s_{n}]].$$

\subsection{Fock space}\label{Nakajima ops}

In this section, we recall details from \cite{hilban} on the Fock space model for $H_{T}^{\ast}(\Hb(\An),\QQ)$, first studied in \cite{nakajima,grojnowski,qin-wang}.  The graded sum
$$\mathcal{F}_{\An}=\bigoplus_{m\geq0} H_{T}^{\ast}(\Hb_{m}(\An),\QQ)$$
can be given (after extension of scalars) the structure of an irreducible representation of a certain Heisenberg algebra $\mathcal{H}$ constructed from $H_{T}^{*}(\An,\QQ)$ as follows.
The algebra $\mathcal{H}$ is generated over $\QQ(t_1,t_2)$
by a central
element $c$ and elements
$$\mathfrak{p}_{k}(\gamma), \quad \gamma \in H_{T}^{\ast}(\An, \QQ),\quad k \in \ZZ, \quad k \ne 0,$$
so that $\mathfrak{p}_{k}(\gamma)$ are $\QQ(t_1,t_2)$-linear in the labels
$\gamma$. The Lie algebra structure on $\mathcal{H}$ is defined by
the following commutation relations
\begin{gather*}
 [\pp_k(\gamma_{1}), \pp_l(\gamma_{2})] = -k \delta_{k+l}\langle\gamma_1, \gamma_2\rangle
\cdot c, \\
 [c, \pp_k(\gamma)] = 0.
\end{gather*}

There are two natural bases for $\mathcal{F}_{\An}$.  The first, already discussed, is the Nakajima basis given by cohomology-weighted partitions.  Given a basis of $H_{T}^{*}(\An,\QQ)$, there is an associated Nakajima basis obtained by taking cohomology-weighted partitions with labels given by basis elements.
We will mainly work with the Nakajima basis with labels lying in the basis
$$\{1, \omega_{1}, \dots, \omega_{n}\}.$$

The second natural basis, obtained after extension of scalars, is given by
$T$-fixed points on $\Hb_{m}(\An)$. 
Given an $(n+1)$-tuple $\overrightarrow{\rho}$ of partitions $\rho^{1}, \dots, \rho^{n+1}$ such that $\sum |\rho_{i}| = m$, we associate the following fixed point $J_{\overrightarrow{\rho}}$ of $\Hb_m(\An)$.
First, given a partition $\rho$, we have the monomial ideal $I_{\rho} \subset \CC[x,y]$ defined by
$$I_{\rho} =(x^{\rho_1},yx^{\rho_2},\dots,y^{l-1}x^{\rho_l})$$
At each fixed point $p_{i}$ of $\An$, 
we can identify the $T$-fixed affine chart centered at $p_i$ with $\CC^2$ by identifying the weights of 
$x$ and $y$ with
$w_{L}^{i}$ and $w_{R}^{i}$ respectively; the restriction of the associated subscheme
to this chart is given by the monomial ideal $I_{\rho^{i}}$ associated to $\rho^{i}$.  The
associated cohomology class $[J_{\overrightarrow{\rho}}]$ give us a basis indexed by multipartitions of $m$.

 Finally, we have the inner product on $\mathcal{F}_{\An}\otimes\QQ(t_1,t_2)$
given by the Poincare pairing on $\Hb(\An)$, defined by $T$-equivariant residue.
We will denote this inner product by undecorated angle brackets
$$\langle \overrightarrow{\mu} | \overrightarrow{\nu}\rangle.$$
We point out here that all these objects can be defined over the ring $R = \QQ[t_1,t_2]_{(t_1+t_2)}$ of rational functions without poles along $t_1+t_2$.

\subsection{Affine algebra operators}

Consider the following affine algebra $\gotg = \widehat{\mathfrak{gl}}(n+1)$ (defined over $\QQ$).  It is generated by
elements
$$x(k)= x\cdot t^{k}, \quad x \in \mathfrak{gl}(n+1),\quad k \in \ZZ,$$
a central element $c$ and a differential $d$.  The defining
relations are
\begin{gather*}
[x(k),y(l)]=[x,y](k+l)+k\delta_{k+l,0}tr(xy)c,\\
[d,x(k)]=kx(k),\quad [d,c]=0,
\end{gather*}
where $tr(xy)$ refers to the trace of the matrix $xy$. 

The Cartan subalgebra of $\gotg$
is given by $\mathfrak{h}\oplus \QQ c\oplus \QQ d$, where  $\mathfrak{h}$ is the
Cartan subalgebra of diagonal matrices.
There are distinguished weights $\Lambda, \delta$ characterized by
$$\Lambda(\mathfrak{h}) = \delta(\mathfrak{h})=0,\quad \Lambda(c)=\delta(d)=1,\quad \Lambda(d)=\delta(c)=0.$$

There is a natural embedding $\mathcal{H}\hookrightarrow\gotg\otimes\QQ(t_1,t_2)$ under which
$\mathcal{F}_{\An}$ admits the following description.  Consider the irreducible representation
$V_{\Lambda}$ of $\gotg$ with highest weight $\Lambda$.  There is a graded isomorphism of $\mathcal{H}$-modules
$$\mathcal{F}_{\An}\otimes\QQ(t_1,t_2) = \bigoplus_{m\geq0} V_\Lambda[\Lambda - m\delta]\otimes\QQ(t_1,t_2),$$ where $V_{\Lambda}[\Lambda-m\delta]$ denotes the weight space for the weight $\Lambda-m\delta$.

This identification allows us to construct operators on $\mathcal{F}_{\An}$.  Given any element 
of the universal enveloping algebra $U(\gotg)$ which commutes with the Cartan subalgebra, the 
associated operator on $V_{\Lambda}$ will preserve the individual weight spaces.   In particular, if we let $e_{ij} \in \mathfrak{gl}(n+1)$ denote the matrix with a $1$ at the entry $(i,j)$ and $0$ everywhere else, then the expression
$$e_{ij}(k)e_{ji}(-k)$$
is a well-defined operator on Fock space.

\subsection{Operator formula}

We define an operator $\Theta^{\DT}(q,s_{1},\dots,s_{n})$ on Fock space using the rubber partition function defined earlier.  If $\overrightarrow{\mu},\overrightarrow{\nu}$ are partitions of $m$, we define
$$\langle \overrightarrow{\mu}| \Theta^{\DT}|\overrightarrow{\nu}\rangle
= q^{-m}\langle \overrightarrow{\mu}, \overrightarrow{\nu}\rangle^{\DT,\sim}_{+}.$$
Again, the brackets on the left-hand side denote the inner product.
We have not included a degree $0$ normalization term.

Also, consider the operator-valued function
of $q,s_{1},\dots,s_{n}$ defined using the affine algebra action:
\begin{gather*}
\Omega_{+}:=\sum_{1\le i<j\le n+1}\sum_{k\in\ZZ} :e_{ji}(k)
e_{ij}(-k):\log(1-(-q)^k s_i\dots s_{j-1})
\end{gather*}
In this expression, we use the normal ordering shorthand where
$$
:e_{ji}(k) e_{ij}(-k):=\left\{ \begin{aligned}
e_{ji}(k)e_{ij}(-k),& \quad k<0 \textrm{ or } k=0, i<j \\
 e_{ij}(-k)e_{ji}(k),& \quad\textrm{otherwise}.
\end{aligned}
\right.
$$
Moreover, we expand the logarithms so the Taylor expansion has
nonnegative exponents in the $s$ variables.

We have the following evaluation.
\begin{theorem}\label{dt theorem}
$$\Theta^{\DT}(q,s_{1},\dots,s_{n}) = (t_1+t_2)\cdot(\Omega_{+}(q,s_{1},\dots,s_{n}) + 
\sum_{1\le i< j \le n+1} F(q, s_{i}\dots s_{j-1})\cdot \mathrm{Id}),$$
where
$$F(q,s) = \sum_{k \geq 0}(k+1)\log(1 - (-q)^{k+1}s).$$
\end{theorem}

In section \ref{proof of the main theorems}, we will explain how this yields the partition functions that occur in theorem \ref{gw dt theorem}.
The proof of Theorem \ref{dt theorem} will be given after first establishing some preliminary results.

\section{Geometric preliminaries}\label{geometry}

In this section, we prove some basic divisibility statements for the equivariant weights associated
to localization and study the geometry of $T$-fixed points of the rubber moduli spaces.  The main result of this section will be the result that the $T$-equivariant DT invariants of $\An\times\mathbf{P}^1$
with nonzero curve class $\beta$
 - absolute, relative, and rubber - have positive valuation with respect to $(t_1+t_2)$.
 
 For the Gromov-Witten theory of $\An\times\mathbf{P}^1$ and the
 quantum cohomology of $\Hb(\An)$, the analog of this statement can be proven
 using the construction of a reduced virtual class from the holomorphic symplectic form on $\An$ (since such a form has weight $-(t_1+t_2)$).  The foundations of this approach in DT theory have yet to appear, so we make a direct but tedious analysis of the localization of the virtual class, following the analogous argument in \cite{okpandt}.
 
\subsection{Descendent insertions}

We recall the definition of descendent insertions in DT-theory.  
For absolute invariants, the universal ideal sheaf
$$\mathcal{I} \rightarrow I_{\chi}(X, (\beta,m)) \times X$$
has a finite $T$-equivariant resolution by locally free sheaves and therefore admits
well-defined $T$-equivariant Chern classes.

For $\gamma\in H_T^l(X,\QQ)$, let $ch_{k+2}(\gamma)$ denote
the following operation on the homology of $I_\chi(X, (\beta,m))$:
\begin{gather*}
ch_{k+2}(\gamma): H^T_*(I_n(X,(\beta,m)),\QQ)\to
H^T_{*-2k+2-l}(I_n(X,(\beta,m)),\QQ),\\
ch_{k+2}(\gamma)(\zeta)=\pi_{1*}(ch_{k+2}(\mathcal{I})\cdot\pi_2^*(\gamma)
\cap\pi_1^*(\zeta)).
\end{gather*}

Given cohomology-weighted partitions of $m$, $\overrightarrow{\nu_1},\dots, \overrightarrow{\nu_b}$, we define relative descendent invariants by the residue integral
\begin{equation*}
\langle\sigma_{k_1}(\gamma_{l_1})\dots\sigma_{k_a}(\gamma_{l_a})
|\overrightarrow{\nu_1},\dots,\overrightarrow{\nu_b}\rangle^{\DT}_{(\beta,m),\chi}:=
\int_{[I_\chi(X/S,(\beta,m))^T]^{vir}}\frac{\prod_{i=1}^a
ch_{k_i+2}(\gamma_{l_i})}{e(Norm^{vir})}\prod_j^b\epsilon_{j}^{\ast}(\overrightarrow{\nu_j}),
\end{equation*}
where
\begin{equation*}
\frac{\prod_{i=1}^a
ch_{k_i+2}(\gamma_{l_i})}{e(Norm^{vir})}:=ch_{k_1+2}(\gamma_{l_1})\circ\dots\circ
ch_{k_a+2}(\gamma_{l_a})\left(\frac{[I_\chi(X,(\beta,m))]^{vir}}{e(Norm^{vir})}\prod_j^b\epsilon_{j}^{\ast}
(\overrightarrow{\nu_j})\right).
\end{equation*}

If the number of relative fibers is nonempty, we can suppress the $m$ in our notation above.
We again assemble the invariants in the generating series, following similar bracket conventions to section~\ref{rubber geometry}:
\begin{align*}
\langle\sigma_{k_1}(\gamma_{l_1})\dots\sigma_{k_a}(\gamma_{l_a})
|\overrightarrow{\nu_1},\dots,\overrightarrow{\nu_b}\rangle^{\DT}_{\beta}&=
\sum_{\chi} q^{\chi}\langle\sigma_{k_1}(\gamma_{l_1})\dots\sigma_{k_a}(\gamma_{l_a})
|\overrightarrow{\nu_1},\dots,\overrightarrow{\nu_b}\rangle^{\DT}_{\beta,\chi}\\
\langle\sigma_{k_1}(\gamma_{l_1})\dots\sigma_{k_a}(\gamma_{l_a})
|\overrightarrow{\nu_1},\dots,\overrightarrow{\nu_b}\rangle^{\DT}&= 
\sum_{\beta} s^{\beta} \langle\sigma_{k_1}(\gamma_{l_1})\dots\sigma_{k_a}(\gamma_{l_a})
|\overrightarrow{\nu_1},\dots,\overrightarrow{\nu_b}\rangle^{\DT}_{\beta}\\
\langle\sigma_{k_1}(\gamma_{l_1})\dots\sigma_{k_a}(\gamma_{l_a})
|\overrightarrow{\nu_1},\dots,\overrightarrow{\nu_b}\rangle^{\DT,\prime}_{\beta}&=
\frac{\langle\sigma_{k_1}(\gamma_{l_1})\dots\sigma_{k_a}(\gamma_{l_a})
|\overrightarrow{\nu_1},\dots,\overrightarrow{\nu_b}\rangle^{\DT}_{\beta}}
{\mathsf{Z}_{DT}(X)_{(0,0), \emptyset, \dots, \emptyset}}.
\end{align*}

We show that, for $\beta \ne 0$, these invariants vanish $\mod (t_1+t_2)$:
\begin{prop}\label{divisibility} The invariants
$$\langle \sigma_{k_1}(\gamma_{l_1}),\dots,
\sigma_{k_r}(\gamma_{l_a})|\overrightarrow{\nu_1}, \dots,\overrightarrow{\nu_b}\rangle^{\DT}_{\beta,\chi}
\in \QQ(t_1,t_2)$$
have positive valuation with respect to $t_1+t_2$.  The analogous result for rubber invariants
is also true.
\end{prop}

The case
where $\beta = 0$ is handled by the results of \cite{okpandt} using a factorization result for
relative DT invariants.

\begin{lemma}
The invariants 
$$\langle \sigma_{k_1}(\gamma_{l_1}),\dots,
\sigma_{k_r}(\gamma_{l_a})|\overrightarrow{\nu_1}, \dots,\overrightarrow{\nu_b}\rangle^{\DT}_{\beta = 0,\chi}$$
vanish $\mod (t_1+t_2)$ for $\chi > m$.
\end{lemma}





\subsection{Rigidification}

In this section, we state a rigidification lemma that allows us to write rubber invariants in terms of non-rubber relative invariants with descendent insertions.  These techniques in Gromov-Witten theory are collectively called rubber calculus \cite{tvgw} and, in the case of DT theory, the details are provided in
section $4.9$ of \cite{okpandt}.

In what follows, we let
$$\delta_{0} = \iota_{*}(\omega_{1}+ \dots \omega_{n}) \in H^{\ast}_{T}(\An\times\mathbf{P}^1,\QQ),$$
where $\iota$ is the inclusion of a fiber $\An \hookrightarrow X$.  The significance of this insertion is that for an effective $\beta\ne 0$ we have
$(\omega_{1}+\dots+\omega_{n})\cdot \beta \ne 0$.

\begin{lemma}
Given a divisor $\omega\in H^{2}_{T}(\An,\QQ)$, we have
$$
\langle \overrightarrow{\mu} | \sigma_{0}(\iota_{*}\omega)| \overrightarrow{\nu}\rangle^{\DT}_{\beta}=
(\omega\cdot\beta)\langle \overrightarrow{\mu}, \overrightarrow{\nu}\rangle^{\DT,\sim}_{\beta}.
$$
In particular, we can rigidify by $\delta_0$ to give the identity:
$$\langle \overrightarrow{\mu}| \sigma_{0}(\delta_{0})| \overrightarrow{\nu}\rangle^{\DT}
= \left(\sum s_{k}\frac{\partial}{\partial s_{k}} \right) \langle \overrightarrow{\mu}, \overrightarrow{\nu}\rangle^{\DT,\sim}.$$
\end{lemma}
\begin{proof}
Let $$\pi:\mathcal{R} \rightarrow I_{\chi}(X, (\beta,m))^{\sim}$$ denote the universal target over the rubber moduli space, equipped with the universal ideal sheaf $\mathcal{I}$.
We can view $\mathcal{R}$ as the moduli space of rubber ideal sheaves together with a point $r$ on the target that does not lie on either relative divisor or singular point of the target.  This induces a map $f: \mathcal{R} \rightarrow \An$.

We remove the $\CC^{\ast}$-scaling by requiring the target point $r$
 to lie over $1 \in \mathbf{P}^{1}$, yielding a map
$$\phi: \mathcal{R} \rightarrow I_{\chi}(X/S,(\beta,m)).$$
If $\mathcal{X}$ is the universal target over the non-rubber relative moduli space, $\phi$ factors by the inclusion
$$\mathcal{R} \hookrightarrow \mathcal{X} \rightarrow I_{\chi}(X/S,(\beta,m)).$$
Here, $\mathcal{R}$ is the substack of $\mathcal{X}$ for which the point on the target
lies over $1$ on $\mathbf{P}^{1}$; clearly the universal ideal sheaf on $\mathcal{X}$ restricts to $\mathcal{I}$.
There is also a virtual class on $\mathcal{R}$ which satisfies the compatibilities
$$[\mathcal{R}]^{vir} = \pi^{*}[I_{\chi}(X,(\beta,m))^{\sim}]^{vir} = \phi^{*}[I_{\chi}(X/S,(\beta,m))]^{vir}.$$

A fiber-wise calculation of Chern classes yields
$$
(\omega\cdot\beta)[I_{\chi}(X,(\beta,m))^{\sim}]^{vir} = \pi_{*}(ch_{2}(\mathcal{I})f^{\ast}(\omega)
[\mathcal{R}]^{vir}).$$  A push-pull argument with respect to $\pi$ and $\phi$ 
gives
$$\begin{aligned}
(\omega\cdot\beta)\langle \overrightarrow{\mu}, \overrightarrow{\nu}\rangle^{\DT,\sim}_{\beta} &=
\langle \overrightarrow{\mu}| ch_{2}(\mathcal{I})f^{\ast}(\omega)|\overrightarrow{\nu}\rangle^{\mathcal{R},\sim}_{\beta}\\
&= \langle \overrightarrow{\mu} | \sigma_{0}(\iota_{*}\omega)| \overrightarrow{\nu}\rangle^{\DT}_{\beta}.
\end{aligned}$$

\end{proof} 

\subsection{Equivariant measure}\label{equivariant vertex}

In the case of Proposition~\ref{divisibility} where there are no relative insertions, we can proceed by localization with respect
to the full $(\CC^{*})^3$-action on $\An\times\mathbf{P}^1$.  Let $t_3$ denote the equivariant variable
in the $\mathbf{P}^1$-direction.  The fixed loci are isolated points consisting
of configurations of boxes arranged in the moment polytope for $\An\times\mathbf{P}^1$; the contribution of each locus to the virtual fundamental class can be decomposed into equivariant factors arising from each vertex and edge of the diagram (see \cite{mnop1, mnop2} for a full description).  We will call these diagrams $\An$-box configurations.  

Again following \cite{okpandt}, we show that the equivariant term associated to each individual $\An$-box configuration will be divisible by $t_1+t_2$, under the assumption that $\beta \ne 0$.  More precisely, given an $\An$-box configuration $\pi$, consisting of (possibly infinite) $3$-dimensional
partitions $\pi_{i}$ at each fixed point, we will
give a precise expression for the multiplicity of $t_1+t_2$ as a factor of the 
associated equivariant vertex weight $\mathsf{w}(\pi)$ in terms of the following data.

Given a $3$d partition $\rho$ centered on $\CC^{3}$ with axes given by $z_1,z_2,z_3$,
let
$$\rho^{0} \supseteq \rho^{1}\supseteq \rho^{2} \dots$$
be the (possibly infinite) Young diagrams obtained by taking level sets across the $z_3$-axis.
This sequence eventually stabilizes to the finite limiting partition of $\rho$ along $z_3$.
Given a $2$-dimensional skew Young diagram $\lambda$, allowing infinite legs, we
have the function
$$\mathrm{rk}(\lambda) = \frac{1}{2} \sum_{r \in \ZZ} | c_{r}(\lambda) - c_{r+1}(\lambda)|$$
where $c_{r}(\lambda)$ is the number of boxes  $(i,j) \in \lambda$
with content $j - i = r$.  For a finite skew diagram, this function is an integer, given by
the miminal decomposition of $\lambda$ into rim hooks.  In general, the rank is not necessarily integral.
The rank with respect to $z_3$ of $\rho$ is defined by
$$\mathrm{rk}_{t_3}(\rho) = \sum_{k=0}^{\infty} \mathrm{rk}(\rho^{k}/\rho^{k+1}).$$
Since $\rho^{k}$ stabilizes, this sum is well-defined.

The following generalizes the statement of Lemma $6$ in \cite{okpandt}.
\begin{lemma}
The multiplicity of $(t_1+t_2)$ in $\mathsf{w}(\pi)$ is equal to
$$\mathrm{mult}_{t_1+t_2}(\mathsf{w}(\pi)) = \sum_{\mathrm{fixed points}} \mathrm{rk}_{t_3}(\pi_{i}).$$
\end{lemma}
\begin{proof}
The equivariant weight consists of a product of vertex contributions from each $3$d partition $\pi_{i}$
and edge contributions from the $2$d partitions associated to each edge of the toric diagram.
For the vertex contribution associated to $\pi_{i}$, let $\lambda_{1}, \lambda_{2}$ be the limiting edge partitions for the edges along the surface $\An$.  The calculation in \cite{okpandt} essentially gives
that the multiplicity of the vertex term is given by
$$\mathrm{rk}_{t_3}(\pi_{i}) - l(\lambda_{1}) - l(\lambda_{2}),$$
where $l(\lambda)$ is the number of rows of the Young diagram with $t_3$ identified with the $y$-axis.

For the edges, the Poincare polynomial encoding linear factors of the edge term
associated to $\rho$ is given by
$$\frac{z_{3}^{-1}F_{\rho}(z_1, z_2)}{1-z_3^{-1}} - \frac{F_{\rho}(z_{1}z_{3}^{-a},z_{2}z_{3}^{-b})}{1-z_3}$$
where 
$$F_{\rho} = Q_{\rho}(z_1,z_2)+ \frac{1}{z_{1}z_{2}}Q_{\rho}(z_1^{-1},z_{2}^{-1}) + 
\frac{Q_{\rho}(z_1,z_2)Q_{\rho}(z_1^{-1},z_{2}^{-1})(1-z_{1})(1-z_{2})}{z_{1}z_{2}}$$
and $Q_{\rho}$ is the Poincare polynomial of the the partition $\rho$.
For edges along $\mathbf{P}^{1}$, the multiplicity of $t_1+t_2$ is zero.  For edges along $\An$, we
have $a = -2, b=0$ and we want to extract the coefficient of $(z_{1}z_{3})^{k}$ for all $k$.
This is obtained by substituting $z_{3} = z_{1}^{-1}$ and taking the constant term.
This is easily seen to be given by $l(\rho)$.  Summing these contributions with the vertex terms gives the lemma.
\end{proof}

We can prove the proposition for absolute DT invariants.
\begin{lemma}\label{absolutedivisibility}
For $\beta \ne 0$,
$$\langle \sigma_{k_1}(\gamma_{l_1}),\dots,
\sigma_{k_r}(\gamma_{l_a})\rangle^{\DT}_{(\beta,m)} = 0 \mod(t_1+t_2)$$
\end{lemma}
\begin{proof}
Given an $\An$-box configuration with $\beta \ne 0$, we show that $\mathrm{rk}_{t_3}(\pi) \geq 1$.
Indeed, since $\beta \ne 0$, there are at least two vertices with exactly one infinite leg along $\An$.
For such $\pi_i$, at least one of the slices $\pi_i^{(k)}$ has a single infinite leg which forces its rank to be at least $1/2$.  As there are at least two such vertices, we have that the multiplicity of $(t_1+t_2)$
in each fixed locus contribution is at least one.  The descendents play no role in this statement, since they contribute polynomial terms to each fixed locus contribution.  
\end{proof}

\subsection{Inductive strategy}

We define an operator $\mathsf{M}$, acting on the space of multipartitions of $m$ with $n+1$ components, which allows us to trade relative insertions for descendent insertions.  In what follows, given a multipartition $\overrightarrow{\mu}= \mu_{1},\dots,\mu_{n+1}$, we use the shorthand $\sigma(\overrightarrow{\mu})$ for the descendent insertions
$$\prod \sigma_{\mu_i -1}(\iota_{*} [p_{i}]).$$
Here
$\iota:\An \rightarrow X$ is the inclusion of a fiber.
Given two multipartitions of $m$, $\overrightarrow{\mu}, \overrightarrow{\rho}$,
let $\mathsf{M}$ be the matrix with elements
$$q^{-m}\langle \sigma(\overrightarrow{\mu})| [J_{\overrightarrow{\rho}}]\rangle^{\DT} \in \CC(t_1,t_2)((q))[[s_1,\dots,s_n]],$$
where $[J_{\overrightarrow{\rho}}]$ is the relative condition imposed by the associated $T$-fixed point.

We first have the following lemma.
\begin{lemma}
The specialization 
$$\mathsf{M}_{0} = \mathsf{M}\mid_{q=s_{1}=\dots=s_{n}=0}$$
is well-defined and invertible. 
\end{lemma}
\begin{proof}
Under the specialization $s_1=\dots=s_n=0$, the matrix $\mathsf{M}$ 
decomposes into the direct sum of tensor products of the analogous matrices $M_d$ for $\CC^2$.  For each factor, the statement follows
from section $4.6$ of \cite{okpandt}.
\end{proof}

We have a partial ordering on the set of discrete $(\beta,\chi)$ given by 
$$(\beta^{'},\chi^{'}) \leq (\beta, \chi)$$
if either $\beta - \beta^{'}$ is nonzero and effective or $\beta = \beta^{'}$ and $\chi^{'} \leq \chi$.
The strategy for proving Proposition \ref{divisibility} is to induct on the discrete invariants $(\beta,\chi)$
and prove the statement simultaneously with the following special case.
\begin{lemma}\label{matrixdivisibility} For $\beta \ne 0$,
$$\langle \sigma(\overrightarrow{\mu})|[J_{\overrightarrow{\rho}}]\rangle^{\DT}_{\beta,\chi} = 0 \mod(t_1+t_2),$$
where $\overrightarrow{\mu}$, $\overrightarrow{\rho}$ are as above.
\end{lemma}

Assuming Lemma \ref{matrixdivisibility} for discrete invariants $(\beta^{'},\chi^{'}) < (\beta,\chi)$,
we first show that this implies Proposition \ref{divisibility} for $(\beta^{'},\chi^{'}) < (\beta,\chi)$.
For simplicity, we show this in the case of a single relative fiber.
By the degeneration formula, we have
\begin{multline*}
\langle \prod \sigma_{a_k}(\iota_{\ast}\gamma_{k}), \sigma(\overrightarrow{\mu})\rangle^{\DT}_{\beta^{'},\chi^{'}}
=  \sum _{\overrightarrow{\rho}} \langle \prod \sigma_{a_k}(\iota_{\ast}\gamma_{k})| [J_{\overrightarrow{\rho}}]\rangle_{\beta^{'},\chi^{'}}^{\DT}
\Delta(\overrightarrow{\rho}) \langle [J_{\overrightarrow{\rho}}]| \sigma(\overrightarrow{\mu})\rangle^{\DT}_{\beta=0,\chi=m}
+ \\\sum_{\stackrel{\overrightarrow{\rho}}{(\beta'',\chi'')< (\beta',\chi')}}
 \langle \prod \sigma_{a_k}(\iota_{\ast}\gamma_{k})|
 [J_{\overrightarrow{\rho}}]\rangle_{\beta^{''},\chi^{''}}^{\DT}
\Delta(\overrightarrow{\rho})q^{-m} \langle [J_{\overrightarrow{\rho}}]| \sigma(\overrightarrow{\mu})\rangle^{\DT}_{\beta^{'''},\chi^{'''}}
\end{multline*}
Here, $\Delta(\overrightarrow{\rho}) = \langle [J_{\overrightarrow{\rho}}],[J_{\overrightarrow{\rho}}]\rangle^{-1}$ is the matrix of gluing terms in the degeneration formula.

The left-hand side and the second summand of the right-hand side have positive valuation with respect to $t_1+t_2$ by Lemma \ref{absolutedivisibility} and our assumption of Lemma \ref{matrixdivisibility}.  The invertibility of $\mathsf{M}_{0}$ implies the result for 
$$ \langle \prod \sigma_{a_k}(\iota_{\ast}\gamma_{k})|[J_{\overrightarrow{\rho}}]\rangle_{\beta^{'},\chi^{'}}^{\DT}.
 $$
Since $[J_{\overrightarrow{\rho}}]$ can be expressed in terms of the Nakajima basis with coefficients without poles along $t_1+t_2=0$, this proves the result for relative invariants.

For rubber invariants of the form
$$\langle \overrightarrow{\mu}| \psi_{0}^{a} \psi_{\infty}^{b}| \overrightarrow{\nu}\rangle_{\beta^{'},\chi^{'}},$$
where $\psi_{0},\psi_{\infty}$ are cotangent lines on the moduli space of target degenerations,
the rigidification lemma above and $\psi$-removal lemmas from \cite{okpandt} express
these invariants in terms of relative invariants of the sort already handled.  This concludes the proof of the full statement of Proposition \ref{divisibility} from the partial statement of Lemma \ref{matrixdivisibility}.

To prove Lemma \ref{matrixdivisibility} and Proposition \ref{divisibility}, it remains to prove a statement in the other direction.  That is, given $(\beta,\chi)$, and assuming Proposition \ref{divisibility}
for all $(\beta^{'},\chi^{'}) < (\beta,\chi)$, we want to show Lemma \ref{matrixdivisibility} for 
$(\beta,\chi)$.  This allows us to prove both statements for all $(\beta,\chi)$.
We will prove this direction via a localization argument with respect to the full torus $T \times \CC^{\ast}$. 
We first discuss the fixed loci of the rubber moduli space with respect to the $T$-action.

\subsection{Skewers and twistors}

In this section, we describe components of the $T$-fixed loci
of $I_{\chi}(X,(\beta,m))^{\sim}$ following the discussion of \cite{okpandt}.
As always with rubber geometries, the degeneration and $\CC^{\ast}$-scaling of the target leads to complicated non-isolated fixed loci.
Let $[I] \in I_{\chi}(X,(\beta,m))^{T,\sim}$ be a $T$-fixed ideal sheaf defined on a rubber target fibered by
$\An$ over a chain $C$ of rational curves.  Given a rational component $P$ of the chain, the restriction
$I_P$ of $I$ to $\An\times P$ can be classified according the following two possibilities, depending
on whether $I_P$ is flat over $P$.  The restriction of $I_P$ to the two distinguished fibers of
$\An\times P$ over $0$ and $\infty$ correspond to $T$-fixed points $[J_{\overrightarrow{\lambda}}], [J_{\overrightarrow{\rho}}]$ of $\Hb_m(\An)$.
In each case, the scaling action defines a $T$-representation on the tangent space to
$P$ at the distinguished points $0,\infty$, we have associated (fractional) $T$-weights
$$w_{P,0}, w_{P,\infty}.$$
\begin{enumerate}
\item If $I_P$ is not flat over $P$, we say $I_{P}$ is a \textit{skewer}.  A skewer arises from a $T$-fixed ideal sheaf of the rigid moduli space
$$I_{*}(\An\times\mathbf{P}^{1}, (\beta,m))^{T}.$$
In this case, we must have $[J_{\overrightarrow{\lambda}}] = [J_{\overrightarrow{\rho}}]$.
Since the ideal sheaf is $T$-fixed on the rigid space, we have
$$w_{P,0} = w_{P,\infty} = 0.$$

\item If $I_P$ is flat over $P$, we say $I_P$ is a \textit{twistor}.  A twistor arises 
from a $T$-fixed map
$$[f_{tw}]\in M_{0,\{0,\infty\}}(\Hb_{m}(\An), *)^{T}.$$
If we identify the domain of the curve with $P$, the ideal sheaf $I_P$ is obtained by pulling back
the universal ideal sheaf via $f_{tw}$.  The image of the fundamental class $[P]$ 
under $f_{tw}$ is determined by the Euler characteristic $\chi(\mathcal{O}_{Z_P})$ and
homology class $\beta$.  More specifically, for the basis of $H^{2}(\Hb(\An),\QQ)$
$$D, (1,\omega_{k})$$
defined in \cite{hilban},
we have
\begin{align}\label{twistordegrees}
\mathrm{deg}(f_{tw}^{\ast}(D)) &= \chi(\mathcal{O}_{Z_{P}}) - m, \\
\mathrm{deg}(f_{tw}^{\ast}((1,\omega_k)))& = \omega_{k}\cdot \beta.\nonumber
\end{align}
\end{enumerate}

The tangent weights $w_{P,0}$ and $w_{P,\infty}$ are precisely
the fractional tangent weights associated at $0$ and $\infty$ to 
the $T$-fixed map $f_{tw}$.  It can be calculated explicitly in
terms of $\overrightarrow{\lambda},\overrightarrow{\rho},$ and the above divisor degrees.

\subsection{Results from \cite{okpandt}}

The following analysis of skewer and twistor contributions is proven in \cite{okpandt}.
We first have the following lemma (Lemma $25$ in \cite{okpandt}).  In what follows,
let $$U_{m,\beta,\chi} \subset \overline{M}_{0,\{0,\infty\}}(\Hb_{m}(\An), \gamma)$$
denote the open set of stable maps for which the domain is a chain of rational curves.  
Here $\gamma \in H_{2}(\Hb(\An),\ZZ)$ is the curve class determined by the degrees (\ref{twistordegrees}).
By pullback, $U_{m,\beta,\chi}$ admits an open immersion into $I_{\chi}(X,(\beta,m))^{\sim}$.

\begin{lemma}\label{hilbdtcompare}
The $T$-equivariant Gromov-Witten and Donaldson-Thomas obstruction theories on $U_{m,\beta,\chi}$
inherited from the above embeddings are canonically isomorphic.
\end{lemma}

The next results cited from \cite{okpandt} directly involve the contribution 
of a given fixed locus to the rubber invariant
\begin{equation}\label{twopointrubber}
\langle [J_{\overrightarrow{\lambda}}] , [J_{\overrightarrow{\rho}}]\rangle^{\DT,\sim}_{\beta,\chi}.
\end{equation}

Given $[I] \in I_{\chi}(X,(\beta,m))^{T,\sim}$ a $T$-fixed ideal sheaf on a rubber target over a chain $C$ of rational curves, we associate the following labelled graph $\Gamma_I$.  
The graph is an oriented chain consisting of vertices $V$ associated to 
\begin{enumerate}
\item connected subcurves of skewers over $C$ (which form a distinguished subset $S\subset V$)
\item nodes $s$ of $C$ for which the two incident components are twistors $P,P'$ for which
$$w_{P,s}+ w_{P',s} \ne 0 \mod t_1+t_2$$
\item and the subset of the two marked points $0$ and $\infty$ on $C$ which are incident to twistors.
\end{enumerate}

The edges $E$ correspond to connected subcurves of twistors for which the sum of the tangent weights 
vanish $\mod t_1+t_2$ at all interior nodes.

We label $\Gamma_{I}$ as follows.  First, each vertex can be associated with a $T$-fixed point
$J_{\overrightarrow{\kappa}} \in \Hb_{m}(\An).$  Second, edges and the distinguished vertices $S$ associated to skewers are also decorated with degrees $(\beta_i,\chi_i)$ associated to the 
$\An$-degree and Euler characteristic of the corresponding subscheme.

The resulting labelled chain graph is invariant as $I$ varies in a connected component of the fixed locus
$I_{\chi}(X,(\beta,m))^{T,\sim}$.  Let $G_{\chi,\beta}(\overrightarrow{\lambda},\overrightarrow{\rho})$
denote the set of labelled, oriented graphs obtained in the above manner.  
Given $\Gamma \in G_{\chi,\beta}(\overrightarrow{\lambda},\overrightarrow{\rho})$
we denote by 
$$\langle [J_{\overrightarrow{\lambda}}] | [J_{\overrightarrow{\rho}}]\rangle^{\Gamma,\sim}_{\beta,\chi}$$
the localization contribution to \eqref{twopointrubber} of those connected components with associated graph $\Gamma$.
As explained in \cite{okpandt}, this localization contribution factors into contributions
$R_{v}$
from descendent skewer integrals for each distinguished vertex in $v\in S$ and descendent twistor integrals $R_{e}$ corresponding
to the edges $e\in E$ of $\Gamma$, where the descendents arise from target cotangent lines.  The following divisibility statement is
proven in section $8.3.5$ of \cite{okpandt}; we refer the reader there for more details.

\begin{lemma}\label{skewtwist}
The twistor contributions $R_{e}\in \CC(t_1,t_2)$
have positive valuation with respect to $t_1+t_2$ for all
labels $(\beta_e, \chi_e)$.  Under the assumption of
Proposition \ref{divisibility} for  $(\beta^{'},\chi^{'}) \leq (\beta,\chi)$,
the skewer contributions $R_v$ have positive valuation for
all skewer integrals with label $(\beta_v,\chi_v) \leq (\beta, \chi)$.
\end{lemma}
\begin{proof}
While the proof in \cite{okpandt} is given for $\CC^{2}$, the same argument applies without modification
to our situation.
For twistor contributions, the divisibility is obtained by using Lemma \ref{hilbdtcompare} to compare 
$R_{e}$ with Hilbert scheme integrals.  Instead of \cite{okpanhilb}, we of course apply divisibility statements from \cite{hilban} to get the analogous statement.  For skewer integrals $R_v$ with
 $(\beta_v,\chi_v) \leq (\beta, \chi)$,
 we express $R_v$ in terms of a rubber invariant with discrete invariants $(\beta_v,\chi_v)$
 and skewer/twistor integrals with lower discrete invariants.  Since the rubber integral 
 has positive valuation by our assumption of Proposition \ref{divisibility}, an inductive argument gives the claim.
 \end{proof}

\subsection{Proof of Lemma \ref{matrixdivisibility}}

We now finish the inductive proof of Lemma \ref{matrixdivisibility} for $(\beta,\chi)$ 
under the assumption of Proposition \ref{divisibility} for $(\beta^{'},\chi^{'})< (\beta,\chi)$.
Given 
\begin{equation}\label{matrixfixed}
\langle \sigma(\overrightarrow{\mu})| [J_{\overrightarrow{\rho}}]\rangle^{\DT}_{\beta,\chi},
\end{equation}
if we apply virtual localization with respect to $T \times \CC^{\ast}$, the fixed loci can be decomposed into an $\An$-box configuration in the bulk with discrete invariants $(\beta_1,\chi_1)$ and a 
fixed locus in the rubber moduli space, associated to a graph $\Gamma$, with discrete invariants
$(\beta - \beta_1, \chi -\chi_1 + m)$.
Using lemma \ref{skewtwist} and the equivariant vertex calculations of section \ref{equivariant vertex}, every
fixed locus contributes a term divisible by $t_1+t_2$ with the possible exception
of the case where $(\beta_1,\chi_1) = (0, m)$ and $\Gamma$ consists of a single skewer vertex
labelled with fixed-point $[J_{\overrightarrow{\rho}}]$ and discrete invariants $(\beta,\chi)$, although possibly with some cotangent lines on the skewer integral.

We can encode these skewer contributions in a matrix $\mathsf{D}$ indexed by multipartitions of $m$.  Since we only have a single skewer vertex, the relative conditions at each end of the rubber fixed locus always coincide, so $\mathsf{D}$ is a diagonal matrix (in the fixed-point basis).  As before, let $\mathsf{M}_{0}$ be the matrix with entries given by
$$\langle \sigma(\overrightarrow{\mu}) | [J_{\overrightarrow{\rho}}]\rangle ^{\DT}_{0, \chi=m},$$
which are given by the classical pairing on $\Hb_m(\An)$.
We have shown that the matrix given by (\ref{matrixfixed}) is given by
\begin{equation}\label{matrixfixed2}
\mathsf{M}_{\beta,\chi} = \mathsf{M}_{0} \circ \Delta \circ \mathsf{D} \mod (t_1+t_2),
\end{equation}
where $\Delta$ is the diagonal matrix encoding gluing relative conditions in the degeneration formula.

It suffices to show that the entries of $\mathsf{D}$ are divisible by $t_1+t_2$.  To do this, we study the matrix
$\mathsf{N}$ indexed by multipartitions with entries given by
$$\langle \sigma(\overrightarrow{\mu}), \sigma(\overrightarrow{\nu}) \rangle^{\DT}_{(\beta,m),\chi} = 0 \mod(t_1+t_2),$$
where the vanishing follows from Lemma~\ref{absolutedivisibility}.
By the degeneration formula, ignoring any terms with discrete invariants smaller than $(\beta,\chi)$,
we have
$$\mathsf{N} = \mathsf{M}_{\beta,\chi} \circ \Delta \circ \mathsf{M}_{0}^{tr} + \mathsf{M}_{0} \circ \Delta \circ \mathsf{M}_{\beta,\chi}^{tr}
\mod (t_1+t_2).$$
If we combine this with equation (\ref{matrixfixed2}), we have
$$0 = \mathsf{M}_{0} \circ \Delta \circ \mathsf{D} \circ \Delta \circ \mathsf{M}_{0}^{tr} + \mathsf{M}_{0} \circ \Delta \circ \mathsf{D} \circ \Delta \circ \mathsf{M}_{0}^{tr} \mod(t_1+t_2),$$
where we use the fact that $\mathsf{D}$ is a diagonal matrix.

Since $\mathsf{M}_{0}$ is invertible, this implies the vanishing of $\mathsf{D}$ and concludes the proof.

\section{Proof of Theorem \ref{dt theorem}}

The strategy of proof is the same as the operator calculations from \cite{hilban}.  In that paper,
it is shown that the operator $\Omega_{+}$ on Fock space is uniquely characterized by five intermediate properties, discussed below.  In fact, we give there
an inductive algorithm to reconstruct the operator from the five propositions below.
These properties will be established directly for $\Theta^{\DT}$ in the following series of propositions.

\subsection{List of propositions}\label{list of props}
The first proposition is a factorization statement that allows us to remove parts
 labelled with cohomology class $1$.

\begin{prop}\label{factorization}
We have the factorization
$$\langle \mu(1) \prod \lambda_{i}(\omega_{i})|\Theta^{\DT}| \nu(1) \prod \rho_{i}(\omega_{i})\rangle
= \langle \mu(1)|\nu(1)\rangle \cdot 
\langle \prod\lambda_{i}(\omega_{i})|\Theta^{\DT}| \prod \rho_{i}(\omega_{i})\rangle.$$
\end{prop}

The next two statements are valid for invariants with only divisor labels.

\begin{prop}\label{dimension}
The coefficients of 
$$\langle \prod\lambda_{i}(\omega_{i})|\Theta^{\DT}| \prod \rho_{i}(\omega_{i})\rangle\in\QQ(t_1,t_2)((q))[[s_{1},\dots,s_{n}]]$$
are of the form $\gamma(t_1+t_2)$, $\gamma \in \QQ$.
\end{prop}

For the next two propositions, we fix $1 \le i < j < n+1$ and isolate the contribution
$$\Theta^{\DT}_{[i,j]}$$
associated to the monomials
$$q^{a}s_{i}^{b_{i}}\cdot\dots s_{j-1}^{b_{j-1}}$$
with $b_{i},b_{j-1} \ne 0.$  Geometrically, this corresponds to curve classes 
$\beta = b_{i}E_{i}+\dots+b_{j-1}E_{j-1}.$
We first give a vanishing statement $\mod(t_1+t_2)^2$ 
for certain matrix elements of $\Theta^{\DT}_{[i,j]}$.
More precisely, we claim that the rational functions involved have valuation at least $2$ 
with respect to $t_1+t_2$.

\begin{prop}\label{vanishing}
Given two distinct $(n+1)$-tuples of partitions $\overrightarrow{\lambda}
\neq \overrightarrow{\eta}$ such that either $|\lambda_{i}| = |\eta_{i}|$ or
$|\lambda_{j}| = |\eta_{j}|$ then
$$
\langle  [J_{\overrightarrow{\lambda}}]|\Theta^{\DT}_{[i,j]}| [J_{\overrightarrow{\eta}}]\rangle = 0
\mod (t_1+t_2)^{2}.
$$
\end{prop}

We have two computations of invariants.  We first define the following fixed points 
on $\Hb_{m}(\An)$, defined by
the multipartitions $\overrightarrow{\rho},
\overrightarrow{\theta},\overrightarrow{\kappa},\overrightarrow{\sigma}$:
\begin{gather*}
\rho_i=(m),\quad \rho_k=\emptyset, \mbox{ if } k\ne  i,\\
\theta_i=(1^m), \quad \theta_k=\emptyset, \mbox{ if } k\ne i,\\
\kappa_i=(m-1),\quad \kappa_j=(1),\quad\kappa_k=\emptyset,\mbox{
if } k\ne i,j,\\
\sigma_i=(1^{m-1}),\quad \sigma_j=(1),
\quad\sigma_k=\emptyset,\mbox{ if } k\ne i,j.
\end{gather*}

\begin{prop}\label{exactevaluation}
For these special two-point correlators we have following expression modulo 
$(t_1+t_2)^2$
\begin{gather*}
\langle
[J_{\overrightarrow{\rho}}]|\Theta^{\DT}_{[i,j]}|[J_{\overrightarrow{\kappa}}]\rangle=(-1)^{m-1}(t_1+t_2)
((n+1) t_1)^{2m}\frac{(m!)^2}{m}
\log(1-(-q)^{m-1}s_{ij}),\\
\langle
 [J_{\overrightarrow{\theta}}]|\Theta^{\DT}_{[i,j]}|[J_{\overrightarrow{\sigma}}]\rangle=
(-1)^{m-1}(t_1+t_2)((n+1) t_1)^{2m}\frac{(m!)^2}{m}
\log(1-(-q)^{-m+1}s_{ij}),
\end{gather*}
where $s_{ij}=s_i\cdot\dots \cdot s_{j-1}$ and the logarithm is expanded
with non-negative exponents in $s_{k}$.
\end{prop}

Finally, we specify the action of $\Theta^{\DT}$ on the vacuum vector $|\emptyset\rangle$.
\begin{prop}\label{vacuum}  
The vacuum expectation is given by
$$\langle \emptyset| \Theta^{\DT}|\emptyset\rangle = 
(t_1+t_2)\sum_{1\le i< j \le n+1} F(q, s_{i}\dots s_{j-1})$$
where
$$F(q,s) = \sum_{k \geq 0}(k+1)\log(1 - (-q)^{k+1}s).$$
\end{prop}

\subsection{Reconstruction}

In \cite{hilban}, we prove the following proposition in sections $4$ and $5$.
\begin{prop}\label{reconstruction}
The operator $(t_1+t_2)\Omega_{+}(q,s_{1},\dots,s_{n})$ satisfies
the claims of Propositions \ref{factorization},\ref{dimension},\ref{vanishing},
and \ref{exactevaluation} and has
vacuum expectation
$$\langle \emptyset| (t_1+t_2)\Omega_{+}| \emptyset\rangle = 0.$$
  It is the unique graded operator-valued power series in $q,s_{1},\dots,s_{n}$
  on $\mathcal{F}_{\An}$ which has these properties.
 \end{prop}

We defer the details of the proof to that paper, we sketch the basic idea here. 
The first part follows from a direct operator calculation using commutation relations of
$\gotg$.  For the uniqueness claim, the argument is to induct on $m$, using the
vacuum expectation as the base case.
 From Propositions \ref{factorization} and \ref{dimension}, it suffices to handle only partitions with divisor labels and, in
this case, to calculate the invariants $\mod (t_1+t_2)^{2}$.  Finally, a symmetric function argument
shows how to reduce these invariants to the two evaluations of Proposition \ref{exactevaluation}
using the vanishing of Proposition \ref{vanishing}.

Assuming the propositions listed above, we explain the proof of Theorem \ref{dt theorem}.
\begin{proof}
If we consider the operator 
$$\Theta^{\DT}(q,s_{1},\dots,s_{n}) - \langle \emptyset|\Theta^{\DT}|\emptyset\rangle\cdot \mathrm{Id},$$
it is easy to see that it still satisfies 
Propositions \ref{factorization}, \ref{dimension},\ref{vanishing} and \ref{exactevaluation}.
Moreover, by construction, its vacuum expectation vanishes.  By Proposition \ref{reconstruction},
we have
$$\Theta^{\DT} = \Omega_{+} +\langle \emptyset|\Theta^{\DT}|\emptyset\rangle\cdot \mathrm{Id}.$$
The calculation of Proposition \ref{vacuum} finishes the result.
\end{proof}

It remains to prove the propositions of section \ref{list of props}.

\section{Factorization}

\subsection{Toric compactification}

It will be useful in this section to work with toric compactifications of $\An$ for which 
effective linear combinations of the $[E_i]$ are rigid curves, linearly independent from the curve classes of the boundary.  While the choice of model is irrelevant for us, we give a specific construction to show they exist. 

We inductively construct projective toric surfaces $S_n$ which contain a $T$-equivariant embedding of $\An$ followed by a $-1$ curve.  For $n=1$, this follows by taking the Hirzebruch surface $F_2$
and blowing up one of the two torus fixed points not on the $-2$-section.  Given $S_n$, we construct $S_{n+1}$ by blowing up the torus fixed point of the $-1$ curve that does not lie on $\An$.  The condition
on curve classes $E_i$ is obvious in this case.

We will study invariants on $\An$ by taking the associated invariants on a compactification $S$
and localizing with respect to $T$.  This will decompose the compactified invariant into residue contributions from $\An$ and contributions from the toric affine chart $\CC^{2}_{z}$ centered on the fixed points $z$ that lie on the boundary of $S$.  In order to distinguish the surfaces, we will label both DT brackets and Fock space brackets with 
the surface being considered - either $\An$, $S$, or $\CC^{2}_{z}$.  

\subsection{Proof of Proposition \ref{factorization}}

Since the statement of the proposition is not affected by taking derivatives, we use the rigidification argument to work with non-rubber relative invariants with a $\sigma_{0}(\delta_0)$ insertion.

We prove the claim by induction on 
$$l = \mathrm{min}(l(\mu), l(\nu)).$$
In the case of $l = 0$, we can assume $\mu = \emptyset$.  We study the associated DT invariant on the compactified threefold $S\times\mathbf{P}^1$ by localization to get
$$\langle \prod \lambda_{i}(\omega_{i})| \sigma_{0}(\delta_0) | \nu(1)\prod \rho_{i}(\omega_{i})\rangle^{\DT,S}_{\beta} = \langle \prod \lambda_{i}(\omega_{i})| \sigma_{0}(\delta_{0}) | \nu(1)\prod \rho_{i}(\omega_{i})\rangle^{\DT,\An}_{\beta} \cdot \prod_{z} \langle \emptyset | \emptyset \rangle^{\DT,\CC^{2}_{z}}_{0}.
$$
This is the only contribution to the localization term since the left insertion and the curve class $\beta$ are forced to lie along $\An$.
It follows from the degree $0$ calculation on $\CC^{2}$ that
$$\langle \emptyset| \emptyset\rangle^{\DT,\CC^{2}}_{0} = 1.$$
By compactness, the invariants of the left-hand side lie in $\QQ[t_1,t_2]$
and have cohomological degree $1 -l(\nu)$; however, we know that from the right-hand side that they are divisible by $(t_1+t_2)$ by Proposition \ref{divisibility}.  Therefore, we must have $\nu = \emptyset$ and the statement of the proposition is vacuous.

In the case where $l > 0$, the localization of the compactified invariant is more complicated.  While the curve class $\beta$ is forced to lie on $\An$, the relative insertions from $\mu$ and $\nu$ (and the degree in the $\mathbf{P}^1$ direction) can lie on the boundary of $S$.
In what follows, we sum over decompositions of $\mu$
$$\mu = \mu_{0} \cup \bigcup \mu_{z}$$
into partitions supported along $\An$ and the fixed points $z$ on the boundary of $S$.

The same localization argument gives the following relation between the invariants on $\An$ and those on $S$:
\begin{multline*}
\langle \mu(1)\prod \lambda_{i}(\omega_{i})| \sigma_{0}(\delta_{0}) | \nu(1)\prod \rho_{i}(\omega_{i})\rangle^{\DT,S}_{\beta}=
\langle \mu(1)\prod \lambda_{i}(\omega_{i})| \sigma_{0}(\delta_{0}) | \nu(1)\prod \rho_{i}(\omega_{i})\rangle^{\DT,\An}_{\beta} + \\
\sum_{\mu_{z},\nu_{z}} C_{\mu_{z},\nu_{z}}
 \langle \mu_{0}(1)\prod \lambda_{i}(\omega_{i})| \sigma_{0}(\delta_{0}) |
 \nu_{0}(1)\prod \rho_{i}(\omega_{i})\rangle^{\DT,\An}_{\beta} \cdot
 \prod_{z} \langle \mu_{z}(1)| \nu_{z}(1)\rangle^{\DT,\CC^{2}_{z}}
\end{multline*}
where the coefficients $C_{\mu_{z},\nu_{z}}$ arise from ordering the parts of $\mu$ and $\nu$.  

First, it follows from Lemma $4$ in \cite{okpandt} that the invariants on $\CC^{2}$ are given
by the Poincare pairing on $\Hb(\CC^{2})$, which is the same as the Fock space pairing.
$$
\langle \mu| \nu\rangle^{\DT,\CC^{2}} = q^{|\mu|} \langle \mu| \nu \rangle^{\CC^{2}}
$$
Also,the $\DT$-invariant on $S$ vanishes for dimension reasons, since $l > 0$.  Except for the leading term, we can apply the inductive hypothesis to every term on the right-hand side.  The result is
\begin{multline*}
0 =  \langle \mu(1)\prod \lambda_{i}(\omega_{i})| \sigma_{0}(\delta_{k}) | \nu(1)\prod \rho_{i}(\omega_{i})\rangle^{\DT,\An}_{\beta} + 
\langle\prod \lambda_{i}(\omega_{i})| \sigma_{0}(\delta_{k}) | \prod \rho_{i}(\omega_{i})\rangle^{\DT,\An}_{\beta}\cdot\\
q^{m}\left(
\sum_{\mu_{z},\nu_{z}} C_{\mu_{z},\nu_{z}} \langle \mu_{0}(1)| \nu_{0}(1)\rangle^{\An}\cdot
  \prod_{z} \langle \mu_{z}(1)| \nu_{z}(1)\rangle^{\CC^{2}_{z}}\right).
\end{multline*}

To finish the argument, we consider the following identity obtained by calculating the Poincare pairing on $\Hb(S)$ via localization:
$$0 = \langle \mu(1)| \nu(1)\rangle^{S} = 
\langle \mu(1)| \nu(1)\rangle^{\An} + 
\sum_{\mu_{z}, \nu_{z}} C_{\mu_{z},\nu_{z}}\langle \mu_{0}(1)| \nu_{0}(1)\rangle^{\An}
\cdot \prod_{z} \langle \mu_{z}(1)|\nu_{z}(1)\rangle^{\CC^{2}_{z}}.$$
Again, the vanishing follows for dimension reasons.  Combining the two equations gives the
factorization.

\subsection{Dimension statement}

For the dimension statement in Proposition \ref{dimension}, we again use the compactification argument.
Indeed, the same calculation from before now shows that
$$\langle \prod \lambda_{i}(\omega_{i})|\sigma_{0}(\delta_{0})|\prod \rho_{i}(\omega_{i})\rangle^{\DT, S}_{\beta}
= \langle \prod \lambda_{i}(\omega_{i})| \sigma_{0}(\delta_{0}) | \prod \rho_{i}(\omega_{i})\rangle^{\DT,
\An}_{\beta}$$
since all insertions are forced to lie along $\An$.  By the dimension formula and compactness, 
the coefficients 
of the left-hand side are linear polynomials in $\QQ[t_1,t_2]$.
Therefore, using the $t_1+t_2$-divisibility of the right-hand side, the coefficients of the overall expression are of the form
$\gamma\cdot(t_1+t_2), \gamma \in \QQ$.

\section{Vanishing and computations}\label{proof of vanishing}

In this section, we use the details of rubber localization discussed in section \ref{geometry} 
for calculations $\mod(t_1+t_2)^{2}$.  In each case, the key idea is to compare
the invariants with results from the geometry of $\Hb_{m}(\An)$.
\subsection{Proof of Proposition \ref{vanishing}}

Recall that we restrict to a curve class $\beta$ with support $[i,j]$ 
and multipartitions $\overrightarrow{\lambda}\neq \overrightarrow{\eta}$ such that
either $|\lambda_{i}| = |\eta_{i}|$ or
$|\lambda_{j}| = |\eta_{j}|$, and we want to prove
$$
\langle  [J_{\overrightarrow{\lambda}}]| [J_{\overrightarrow{\eta}}]\rangle^{\DT,\sim}_{\beta} = 0 \mod(t_1+t_2)^{2}.$$

It follows from Lemma \ref{skewtwist} that we can ignore contributions from 
graphs $\Gamma \in G_{\chi,\beta}(\overrightarrow{\lambda},\overrightarrow{\eta})$
which have skewer components or multiple edges.  Therefore, it suffices to only consider graphs $\Gamma$ which correspond to $T$-fixed stable maps to $\Hb_{m}(\An)$ for which
the domain is a chain of rational curves and the tangent weights at each node sum to $0 \mod t_1+t_2$.  These are known as \textit{unbroken} maps (see section~4 of \cite{hilban}).  We show the following in Proposition~4.9 of that paper.

\begin{lemma}
There are no unbroken maps connecting $[J_{\overrightarrow{\lambda}}]$ to $[J_{\overrightarrow{\eta}}]$
with curve class with support $[i,j]$.
\end{lemma}
This immediately gives the desired vanishing.

\subsection{Proof of Proposition \ref{exactevaluation}}\label{proof of vanishing for not root}

We just handle the evaluation
$$\langle [J_{\overrightarrow{\rho}}]|\Theta^{\DT}_{[i,j]}|[J_{\overrightarrow{\kappa}}]\rangle 
\mod (t_1+t_2)^{2};$$
 the proof of the other invariant is identical.
We follow the same argument as in the last section.  Again, we use lemma \ref{skewtwist}
to only consider graphs $\Gamma$ corresponding to unbroken $T$-fixed maps
to $\Hb_m(\An)$ with support $[i,j]$. The following lemma is equivalent to the Lemma~4.11 from \cite{hilban}.
\begin{lemma}
Given an unbroken map connecting fixed points 
$J_{\overrightarrow{\rho}}$ and $J_{\overrightarrow{\kappa}}$ with curve class $\gamma$, there exists $d>0$
such that $$(1,\omega_{l})\cdot \gamma = d \mbox{ for } i\leq l \leq j-1$$
and $$(1,\omega_{l})\cdot\gamma = 0$$ otherwise.  In this case, the
unbroken map is unique.
\end{lemma}

In our situation, this implies that only $\beta = d\alpha_{i,j}$
contributes to the invariant $\mod (t_1+t_2)^{2}$ and there is a unique 
graph $\Gamma$ that contributes as well.
It then follows from Lemma \ref{hilbdtcompare} that the contribution of $\Gamma$
is the same as the contribution of the associated unbroken map to the two-point invariant
on $\Hb_{m}(\An)$.  From the calculation of Proposition~4.3 in \cite{hilban}, we have
$$\langle [J_{\overrightarrow{\rho}}]| [J_{\overrightarrow{\kappa}}]\rangle^{\DT,\sim}_{\beta} =  
(-1)^{m-1}(t_1+t_2)((n+1)t_1)^{2m}\frac{1}{d}\frac{(m!)^{2}}{m}
(-1)^{md+m+d}q^{md+m-d}\mod (t_1+t_2)^{2}.$$
The shift in the $q$-variable exactly corresponds to the discrepancy between curve degree on $\Hb(\An)$ and the Euler characteristic that appears in equation
(\ref{twistordegrees}).

\section{Vacuum expectation}

The vacuum expectation is the only calculation in this paper which cannot be reduced formally to a calculation on the Hilbert scheme of points on $\An$.  

After rigidification, it suffices to prove the following lemma
 \begin{lemma}
 $$\langle \emptyset | \sigma_{0}(\delta_0)|\emptyset\rangle^{\DT,\prime}_{\beta} = 
 -(t_1+t_2)(j-i)\frac{(-q)^{d}}{(1-(-q)^{d})^{2}}
 $$
 if $\beta = d\alpha_{i,j}$ where $\alpha_{i,j}$ are the root curve classes of
 section \ref{proof of vanishing for not root}.  Otherwise the invariant vanishes.
 \end{lemma}
 
 \subsection{Minimal configurations}
 
 From Proposition \ref{dimension}, these invariants are proportional to
 $(t_1+t_2)$, so it suffices to calculate them $\mod (t_1+t_2)^{2}$.  Second,
 it follows from the degeneration formula
 and the calculation of the invariants 
 $\langle |\emptyset\rangle$ - whose proof we defer until section \ref{cap and tube} -
 that we can replace the relative invariants with
 $$\langle \sigma_{0}(\delta_0)\rangle^{\DT,\prime}_{(\beta,m=0)}\equiv \langle
 \sigma_0(\delta_0) \rangle^{\DT}_{(\beta,m=0)} \mod (t_1+t_2)^{2}.$$
 This last congruence follows from the fact that the degree $0$ series for
 $\An$ is congruent to $1$ modulo  $(t_1+t_2)$.
 
 This last expression will be handled by localizing with respect to $T \times
 \CC^{\ast}$; we assume
 that $\delta_0$ is supported over $0 \in \mathbf{P}^1$.  Since we can ignore
 factors of $(t_1+t_2)^2$,
 we only consider $\An$-box configurations $\pi$ with
 $\mathrm{mult}_{t_1+t_2}(\pi) = 1$.  

We can enumerate such configurations as follows.  First, assume 
$$\beta = c_{i}[E_i] + c_{i+1}[E_{i+1}]+\dots +c_{j-1}[E_{j-1}]$$
with $c_{i}, c_{j-1} \ne 0$.  Let $\pi_{r}$ denote the $3$d partition concentrated at $p_{r}\times 0$.
So $\pi_{i}, \pi_{j}$ have exactly one infinite leg and $\pi_{k}$ have two infinite legs for $i < k< j$.  

Since we want $\pi$ to have minimal rank, these are the only fixed points with nonempty $3$d partitions.
The lower bounds from section \ref{equivariant vertex} imply that we must have the equalities
$$\mathrm{rk}_{t_3}\pi_{i} = \mathrm{rk}\pi_{j} = \frac{1}{2},\quad \mathrm{rk}_{t_3}(\pi_{k}) = 0,\quad i< k< j.$$

The only $\An$-box configurations $\pi$ with this property can be characterized as follows.  
For $a,b\geq 0$, let $H_{a,b}$ be the $2$d diagram on the $\An$-skeleton such that
\begin{enumerate}
\item at $p_{i}, H_{a,b}$ has a single infinite leg of width $1$ along $w_{R}^{i}$ and $a$ additional boxes in a single column along $w_{L}^{i}$
\item at $p_{j}, H_{a,b}$ has a single infinite leg of width $1$ along $w_{L}^{j}$ and $b$ additional boxes in a single row along $w_{R}^{j}$,
\item at $p_k$ for $i < k < j$, $H_{a,b}$ has infinite legs of width $1$ in both directions and no additional boxes.
\end{enumerate}
Given a minimal box configuration $\pi$, there exists $d \geq 1, a, b \geq 0$, such that the
slices of $\pi$ along $t_{3}$ are
$$\pi^{0} = \pi^{1} = \dots = \pi^{d-1} = H_{a,b}.$$
These configurations will be denoted $H_{a,b}^{d}$
In particular, this forces $\beta = d \alpha_{ij}$ which implies the vanishing statement.

\subsection{Contribution of $H_{a,b}^{d}$}

There are two factors associated to each fixed locus.  First, there is the insertion 
$\sigma_{0}(\delta_{0})$, which gives the factor
$$t_{3}(\sum\omega_k\cdot \beta) = dt_{3}(j-i).$$
Second, there is the localized virtual class $\mathsf{w}(H_{a,b}^{d})$.   Fortunately,
our $\An$-box configurations are extremely simple so the respective vertex
and edge contributions are easily calculated as follows.  We group together all factors of $(t_1+t_2)$
and write the remaining factors $\mod(t_{1}+t_{2})$.  
\begin{multline*}
\mathsf{w}(\pi_{k})=(-1)^d
\frac{dt_3}{(t_1+t_2)}\prod_{r=0}^{d-1}
\frac{(w^R_m+rt_3)^2}{(w^R_m-(r+1)t_3)^2}\\
\frac{(w^L_m+rt_3)^2}{(w^L_m-(r+1)t_3)^2}
\frac{(2w^R_m-(r+1)t_3)(2w^L_m-(r+1)t_3)}{(2w^R_m+rt_3)(2w^L_m+rt_3)},\quad
i < k < j.
\end{multline*}
\begin{gather*}
\mathsf{w}(\mathrm{edge})=\frac{(t_1+t_2)}{-t_3}\prod_{r=1}^{d-1}
\frac{-rt_3}{-(r+1)t_3}\\
\mathsf{w}(\pi_{i})=\prod_{r=0}^{d-1}\prod_{s=1}^a
\frac{((r-d)t_3-s(n+1)t_1)}{(rt_3+s(n+1)t_1)},\quad
\mathsf{w}(\pi_{j})=\prod_{r=0}^{d-1}\prod_{s=1}^b
\frac{((r-d)t_3-s(n+1)t_2)}{(jt_3+s(n+1)t_2)}.
\end{gather*}
Since the overall product of these terms has no pole at $t_3$ and since our ultimate answer 
has no equivariant dependence on $t_3$, we can evaluate the answer by setting $t_3=0$.  This yields
$$\sum_{a,b\geq 0} (t_1+t_2)q^{d(1+a+b)}(-1)^{d(1+a+b)+1}$$
which yields the desired statement.

This completes the proof of Theorem \ref{dt theorem}.

\section{Proofs of Main Theorems}

\subsection{Cap and tube invariant}\label{cap and tube}

In this section, we first calculate the relative DT theory of $X$ relative to $1$ and $2$ fibers, referred to as cap and tube invariants, and prove the main theorems for $\overrightarrow{\rho} = (1)^{m}$.

\begin{prop}\label{cap evaluation}
Let $\overrightarrow{\mu} = \mu_{1}([p_{1}])\dots\mu_{n+1}([p_{n+1}])$
be a cohomology-weighted partition of $m$, with labels given by the fixed-point basis of $H_{T}^{*}(\An,\QQ)$.   We then have
$$\mathsf{Z}^{\prime}_{\DT}(X)_{\overrightarrow{\mu}} = q^{m}\prod_{i=1}^{n+1} 
\frac{\delta_{\mu_{i}, (1)^{m_{i}}}}{m_{i}!},$$
where $m_{i} = |\mu_{i}|$.
\end{prop}
\begin{proof}
The $\beta=0$ contribution to the cap invariant
follows from the cap calculation for $\CC^{2}$ of \cite{okpandt}.  Indeed, since $\beta = 0$, the relative invariants factor over contributions from the toric affine charts $\CC^{2}_{i}$ centered at the fixed points
$p_{i}$:
$$
\mathsf{Z}^{\prime}_{\DT}(X)_{\beta=0,\overrightarrow{\mu}} = \prod_{i=1}^{n+1}
\mathsf{Z}^{\prime}_{\DT}(\CC^{2}_{i}\times \mathbf{P}^1)_{\mu_{i}([p_{i}])}.
$$
Comparison with Lemma $21$ from \cite{okpandt} finishes the claim.

For $\beta \ne 0$, we need to show the contribution vanishes.  Following the proof
of Proposition \ref{factorization}, we take the toric compactification $\An \subset S$. After
normalizing by the degree $0$ partition function, we have the 
equality
$$\mathsf{Z}^{\prime}_{\DT}(X)_{\overrightarrow{\mu}} = \mathsf{Z}^{\prime}_{\DT}(S\times\mathbf{P}^{1})_{\overrightarrow{\mu}}$$
since all relative conditions are fixed off the boundary  this implies via compactness that the invariants with $\beta \ne 0$ lie in $\QQ[t_1,t_2]$ and, by Proposition \ref{divisibility} are divisible by $t_1+t_2.$  However
the invariants have degree
$$-2m + (m-l) + 2l = l-m \leq 0,$$ 
which gives the vanishing.
\end{proof}

\begin{prop}\label{tube}
For tube invariants, we have
$$\mathsf{Z}^{\prime}_{\DT}(X)_{\overrightarrow{\mu},\overrightarrow{\nu}}=\langle \overrightarrow{\mu} , \overrightarrow{\nu}\rangle^{\DT} = q^{m}\langle 
\overrightarrow{\mu}| \overrightarrow{\nu}\rangle$$
\end{prop}

\begin{proof}

This is a general statement for all threefolds of the form $S\times\mathbf{P}^{1}$.
If we define an operator $\mathsf{A}$ on Fock space by the equality
$$\langle \overrightarrow{\mu}| \mathsf{A}|\overrightarrow{\nu}\rangle = 
q^{-m}\langle \overrightarrow{\mu} | \overrightarrow{\nu}\rangle^{\DT},$$
then the degeneration formula applied to the degeneration of $\mathbf{P}^1$ into a chain of two rational
curves gives the identity
$$\mathsf{A} = \mathsf{A}\circ\mathsf{A}.$$
Since we know that $\mathsf{A}$ is invertible, it must be the identity matrix, which
gives the result.
\end{proof}

If we consider the degeneration of $\mathbf{P}^{1}$ with marked points to two rational curves,
with all marked points on a single component, we immediately obtain the following calculation
$$\mathsf{Z}^{\prime}_{\DT}(X)_{\overrightarrow{\mu_{1}},\dots,\overrightarrow{\mu_{k}}} = 
\mathsf{Z}^{\prime}_{\DT}(X)_{\overrightarrow{\mu_{1}},\dots,\overrightarrow{\mu_{k}}, (1)^{m}}.$$ 
In particular, we have
\begin{prop}For $\overrightarrow{\rho} = (1)^{m}$
$$\mathsf{Z}^{\prime}_{\DT}(X)_{\overrightarrow{\mu},\overrightarrow{\rho},\overrightarrow{\nu}} = q^{m}\langle
\overrightarrow{\mu},\overrightarrow{\nu}\rangle.$$
\end{prop}
This implies Theorems \ref{gw dt theorem} and \ref{hilb dt theorem} for the identity element $\overrightarrow{\rho} = (1)^{m}$
by comparison with \cite{gwan},\cite{hilban}.

\subsection{Proofs of Theorems \ref{gw dt theorem} and \ref{hilb dt theorem}}\label{proof of the main theorems}

We will prove these theorems by computing the partition functions via Theorem \ref{dt theorem}
and comparing with the analogous results from \cite{gwan},\cite{hilban}.

Recall the basis of divisors of $H_{T}^{*}(\Hb_{m}(\An),\QQ)$ given by the cohomology-weighted partitions
$$D, (1,\omega_{i}),\quad i = 1, \dots, n.$$
Let 
$M_{D}^{cl}, M_{(1,\omega_{i})}^{cl}$ be the operators on $\mathcal{F}_{\An}$ 
given by classical multiplication by these basis elements.
We define operators $M_{D}(q,s_{1},\dots,s_{n}), M_{(1,\omega_{i})}(q,s_{1},\dots,s_{n})$ 
on $\mathcal{F}_{\An}$ by
the equalities
\begin{gather*}
\langle \overrightarrow{\mu} | M_{D}| \overrightarrow{\nu}\rangle 
= - q^{-m} \mathsf{Z}^{\prime}_{\DT}(X)_{\overrightarrow{\mu}, (2), \overrightarrow{\nu}}\\
\langle \overrightarrow{\mu} | M_{(1,\omega_{i})}| \overrightarrow{\nu}\rangle 
=  q^{-m} \mathsf{Z}^{\prime}_{\DT}(X)_{\overrightarrow{\mu}, (1,\omega_{i}), \overrightarrow{\nu}}.
\end{gather*}

Finally, we define the following operator using the Heisenberg algebra operators 
defined in section~\ref{Nakajima ops}:
$$\Omega_{0} = -\sum_{k\geq 1} \left[(n+1)t_{1}t_{2}\pp_{-k}(1)\pp_{k}(1) +
\sum_{i=1}^{n} \pp_{-k}({E_{i}})
\pp_{k}(\omega_{i})\right]\log\left(\frac{1-(-q)^k}{1-(-q)}\right).$$

The following proposition gives an evaluation for our partition functions.  By comparing it
with Theorems 2.1 and 2.2 in \cite{hilban}, it completes the proofs of Theorems~\ref{gw dt theorem} and \ref{hilb dt theorem}.

For the convenience of the reader let us recall the formula for the operator $\Omega_+$:
\begin{gather*}
\Omega_{+}:=\sum_{1\le i<j\le n+1}\sum_{k\in\ZZ} :e_{ji}(k)
e_{ij}(-k):\log(1-(-q)^k s_i\dots s_{j-1}).
\end{gather*}

\begin{prop}

\begin{gather*}
M_{D} = M_{D}^{cl} + (t_1+t_2) q \frac{\partial}{\partial q}(\Omega_{0} + \Omega_{+})\\
M_{(1,\omega_{i})} = M_{(1,\omega_{i})}^{cl} + (t_1+t_2)s_{i}
\frac{\partial}{\partial s_{i}} \Omega_{+}
\end{gather*}

\end{prop}

\begin{proof}
As before the $\beta = 0$ contributions - both the classical part and $\Omega_{0}$ -
can be deduced from the case of $\CC^{2}$.  The DT invariants factor
into contributions from each fixed point $p_{i}$ and, as we explain in detail in section~6 of 
\cite{hilban}, $\Omega_{0}$ factors into the contribution of the rubber operator for $\CC^{2}$.

For $\beta \ne 0$, we argue as follows in the case of $(1,\omega_{i})$.
Using the rigidification lemma, we have
$$\begin{aligned}
(\omega_{i}\cdot \beta) \langle \overrightarrow{\mu}, &\overrightarrow{\nu}\rangle^{\DT,\sim}_{\beta}
=  \langle \overrightarrow{\mu}| \sigma_{0}(\iota_{*}\omega_{i})| \overrightarrow{\nu}\rangle^{\DT}_{\beta}
\\
&= \sum_{\stackrel{\overrightarrow{\rho}}{\beta_1+\beta_2=\beta}}
\langle \overrightarrow{\mu}, \overrightarrow{\nu}, \overrightarrow{\rho}^{\vee}\rangle^{\DT, \prime}_{\beta_1}
\Delta_{\rho} q^{-m} \langle \overrightarrow{\rho},  \sigma_{0}(\iota_{*}\omega_{i}), (1)^{m}\rangle^{\DT,\prime}_{\beta_2}.
\end{aligned}
$$
Here $\overrightarrow{\rho}^{\vee}$ denotes cohomology-weighted partitions with labels in a 
Poincare-dual basis and $\Delta_{\rho}$ denote combinatorial gluing terms in the degeneration formula.

If $\beta_2 \ne 0$, Proposition \ref{factorization} forces $\overrightarrow{\rho} = (1)^{m}$
which from Proposition \ref{tube} forces $\beta_1 = 0$.  If $\beta_2 = 0$, the second term
is given by classical multiplication on $\Hb_{\An}$.  The resulting expression
is
$$
\langle \overrightarrow{\mu}, (1,\omega_{i}), \overrightarrow{\nu}\rangle^{\DT,\prime}_{\beta}
+ \langle \overrightarrow{\mu}, \overrightarrow{\nu}\rangle 
\langle \emptyset |  \sigma_{0}(\iota_{*}\omega_{i})| \emptyset\rangle^{\DT,\prime}_{\beta}.
$$
Summing over $\beta\ne 0$ gives
$$M_{(1,\omega_{i})} - M_{(1,\omega_{i})}^{cl} = 
s_{i}
\frac{\partial}{\partial s_{i}} \Theta^{\DT} - \langle \emptyset |s_{i}
\frac{\partial}{\partial s_{i}} \Theta^{\DT} |\emptyset\rangle\cdot \mathrm{Id}
$$
which gives the result.

In the case of $D$, the analogous argument applies.  The only difference is that we
rigidify with $-\sigma_{1}(F)$ (here $F$ is a fiber of the projection on $\mathbf{P}^1$) 
which yields $q\frac{\partial}{\partial q} \Theta^{\DT}$.  
\end{proof}

As an immediate corollary of this proposition,
we prove a statement of the GW/DT for primary insertions in the relative theory.
In the equation below, the right-hand side refers to the partition function, given by a
Laurent series in $u,s_{1},\dots,s_{n}$,
 for the relative Gromov-Witten theory of
$\An\times\mathbf{P}^{1}$ with primary insertions.

\begin{cor}
Given divisor classes $\omega_{k_{1}},\dots,\omega_{k_{l}}$, the following partition functions
are equal, after the change of variables $q= -e^{iu}$,
$$
(-q)^{-m}\langle \overrightarrow{\mu}| \prod_{i=1}^{l}\sigma_{0}(\iota_{\ast}\omega_{k_{i}})|\overrightarrow{\nu}\rangle^{\DT}
=(-iu)^{l(\mu)+l(\nu)}
\langle \overrightarrow{\mu}|\prod_{i=1}^{l}\tau_{0}(\iota_{\ast}\omega_{k_{i}})|\overrightarrow{\nu}
\rangle^{\mathrm{GW}}.
$$
In particular, the left-hand side is a rational function of $q,s_{1},\dots,s_{n}$.
\end{cor}

\begin{proof}
Since the $\beta=0$ case can be deduced from \cite{okpandt}, we focus on the $\beta\ne 0$ contribution.  By degenerating the base $\mathbf{P}^{1}$ to a chain of rational curves, with a primary insertion on each component, it suffices to show this for a single primary insertion.  In this case,
the only difference between this evaluation and the operator $M_{(1,\omega_{k})}$
is the vacuum expectation, given by
$$
s_{k}\frac{\partial}{\partial s_{k}}\sum_{1\le i< j \le n+1} F(q, s_{i}\dots s_{j-1})\cdot \mathrm{Id}.$$
This can be matched directly with the vacuum correlator from Proposition $3.6$ in \cite{gwan}.
\end{proof} 

\subsection{Generation conjecture}\label{generation conjecture}

In this section, we state
a nondegeneracy conjecture on the operators $M_{D}, M_{(1,\omega_{i})}$
that will allow us to prove Theorems \ref{gw dt theorem} and \ref{hilb dt theorem} unconditionally.

We first observe that the relative DT theory of $X$ defines a ring structure $\circ$ on
$$H_{T}^{\ast}(\Hb_{m}(\An),\QQ)\otimes\QQ(t_1,t_2)((q,s_{1},\dots,s_{n}))$$
by the equation
$$\langle \overrightarrow{\mu}, \overrightarrow{\rho}\circ \overrightarrow{\nu}\rangle
= q^{-m}\langle \overrightarrow{\mu}, \overrightarrow{\rho}, \overrightarrow{\nu}\rangle^{\DT,'}.$$
It follows from the degeneration formula that $\circ$ defines a graded-commutative, associative product;
moreover it is easy to see that it is a ring deformation of the classical $T$-equivariant cohomology
of $\Hb_{m}(\An)$ with $(1)^{m}$ as a unit (by Proposition \ref{tube}).

In particular, the operators $M_{D}, M_{(1,\omega_{i})}$ commute with each other; in \cite{hilban}, we make the following conjecture about their joint spectrum.

\begin{conj}\label{generation conj}
The joint eigenspaces for the operators $M_{D},
M_{(1,\omega_{i})}$ are one-dimensional for all $m>0$.
\end{conj}
Although we are unable to prove this conjecture, we can prove the suggestive statement
that $M_{D}(q,s_{1},\dots,s_{n}) - M_{D}^{cl}$ has distinct eigenvalues.

Assuming this conjecture, we have the following immediate corollary.

\begin{corstar}
Under the above conjecture,
 the divisor classes $D, (1,\omega_{i})$ generate the DT-ring deformation of $H_{T}^{\ast}(\Hb(\An),\QQ)$ over $\QQ(t_1,t_2,q,s_{1},\dots,s_{n})$.
Moreover, the DT product $\circ$ is identical to the quantum product defined by the
small quantum cohomology of $\Hb_{m}(\An)$.
\end{corstar}
\begin{proof}
For the first claim, we observe that the joint eigenvectors for $M_{D}, M_{(1,\omega_{i})}$
are idempotents for the semisimple DT-ring deformation.  It then follows from the Vandermonde formula that the vectors
$$M_{D}^{a}\prod_{i} M_{(1,\omega_{i})}^{b_{i}} (1^{m}),\quad a, b_{i}\geq 0$$
have full span.
For the second claim, since the operators here coincide with quantum multiplication by
divisors $D, (1,\omega_{i})$ and these multiplication operators generate the whole ring, it is clear
that quantum multiplication by an element $\overrightarrow{\rho}$ coincides with the
operator $\overrightarrow{\rho}\circ$.
\end{proof}
Since the structure constants of quantum cohomology are given by three-point functions, this completes the proof of Theorem \ref{hilb dt theorem} under the assumption of the generation conjecture.

We now give the proof of $\mathrm{Theorem}^{*}$~\ref{gw dt with conjecture}.
\begin{proof}

We first handle the case of $k \leq 3$ relative fibers using the technique above.
For the rationality claim, given $\overrightarrow{\rho}$, it follows from the generation conjecture that there exist rational functions
$c_{a,b_{1},\dots,b_{n}}(q,s_{1},\dots,s_{n})$ such that
$$\overrightarrow{\rho} = \sum_{a,b_{1},\dots,b_n} c_{a,b_{1},\dots,b_n}M_{D}^{a}\prod_{i} M_{(1,\omega_{i})}^{b_{i}} (1^{m}).$$
Therefore, we have the operator equality
$$M_{\overrightarrow{\rho}} = \sum_{a,b_{1},\dots,b_n} c_{a,b_{1},\dots,b_n} M_{D}^{a}\prod_{i} M_{(1,\omega_{i})}^{b_{i}}$$
which implies that the matrix elements of the left-hand side are rational functions in $q,s_{1},\dots,s_{n}$.

For the GW/DT matching, as explained in \cite{gwan}, the relative Gromov-Witten theory of $X$ defines a ring structure on
$$H_{T}^{\ast}(\Hb_{m}(\An),\QQ)\otimes\QQ(t_1,t_2)((u,s_{1},\dots,s_{n})).$$
Since divisor operators coincide for both rings, after the change of variables $-q = e^{iu}$, the same
argument given for the small quantum product implies the matching for all three-point functions.

Finally, for $k >3$, there is a degeneration of $\mathbf{P}^{1}$ with $k$ marked points to a
chain of $k-2$ rational curves, each with $3$ marked or nodal points. Since the GW/DT matching is preserved under degeneration, this reduces us to the case of $k=3$.
\end{proof}

The ring deformation $\circ$ of $H_{T}^{\ast}(\Hb(\An),\QQ)$ defined using relative DT theory is valid for any surface $S$.  While we always expect a matching with relative Gromov-Witten theory of
$S\times\mathbf{P}^{1}$,
the naive matching with the small quantum cohomology ring is not true in general.  For instance,
already when $m=1$ for $\mathbf{P}^{2}$, the structure constants for $\circ$
have a nontrivial $q$-dependence while the quantum cohomology structure constants
have no $q$-dependence.  However, we expect that there is a modification that will make
the matching of Theorem \ref{hilb dt theorem} valid for all surfaces $S$.

Finally, we conclude by sketching how to extend the results here to $\An$-bundles over higher genus curves.  Given the rank $2$ bundle $\mathcal{O}(a)\oplus\mathcal{O}(b)$ over a curve $C$, let $X(a,b)$ be the threefold obtained via fiberwise quotient and resolution by $\mathbb{Z}_{n+1}$.  It follows from the degeneration formula and the enriched TQFT formalism of \cite{localcurves} that the DT-theory of $X(a,b)$ can be computed in terms of the following pieces.  First, we have the theory of $\An\times\mathbf{P}^{1}$, relative to 1,2, or 3 fibers - computed here under the assumption of the generation conjecture.  Second, we have the theory of $X(0,-1)$ relative to a fiber.  These integrals can be evaluated under the Calabi-Yau specialization $t_1+t_2+t_3=0$, where the results of \cite{mnop1} apply.  For the comparison to Gromov-Witten theory for $X(0,-1)$, this follows from the topological vertex formalism, proven in \cite{lllz, MOOP}.

\vspace{+10 pt}
\noindent 
Department of Mathematics\\
Columbia University\\
New York, NY 10027, USA\\
dmaulik@cpw.math.columbia.edu

\vspace{+10 pt}
\noindent
Department of Mathematics\\
Princeton University\\
Princeton, NJ 08544, USA\\
oblomkov@math.princeton.edu

\end{document}